\documentclass[pdflatex,sn-mathphys-num]{sn-jnl}% Math and Physical Sciences Numbered Reference Style

\usepackage{graphicx} % Required for inserting images

\usepackage{amsmath,amssymb,amsthm}
\usepackage[ruled,vlined]{algorithm2e}
\usepackage{algpseudocode}
\SetKwComment{Comment}{// }{}
\SetKwFunction{FMain}{MainFunction}
\SetKwProg{Fn}{Function}{:}{}
\usepackage{natbib}
\usepackage{tikz}
\usepackage{pgfplots}
\usepackage{booktabs}
\usepackage[pdf]{graphviz}
\usepackage{multirow}
\usepackage{mathdots}
\usepackage{hyperref}
\usepackage[textsize=footnotesize]{todonotes}
\usetikzlibrary{shapes.geometric}
\usetikzlibrary{calc}
\usetikzlibrary{tikzmark,overlay-beamer-styles,babel}
\usetikzlibrary{fit, matrix}
\usetikzlibrary{arrows.meta}
\tikzset{>={Latex[length=3mm, width=2.5mm]}}
\pgfplotsset{compat=1.18}

\tikzset{auto shift/.style={auto=right,
		to path={ let \p1=(\tikztostart),\p2=(\tikztotarget),
			\n1={atan2(\y2-\y1,\x2-\x1)},\n2={\n1+180}
			in ($(\tikztostart.{\n1})!1mm!270:(\tikztotarget.{\n2})$) -- 
			($(\tikztotarget.{\n2})!1mm!90:(\tikztostart.{\n1})$) \tikztonodes}}}

\newtheorem{definition}{Definition}
\newtheorem{theorem}{Theorem}
\newtheorem{prop}[theorem]{Proposition}
\newtheorem{remark}{Remark}
\newtheorem{corollary}[theorem]{Corollary}
\newtheorem{lemma}[theorem]{Lemma}

\newcommand{\R}{\mathbb{R}} % Real #s
\newcommand{\Z}{\mathbb{Z}} % Integer #s
\DeclareMathOperator{\Gcal}{\mathcal{G}}
\DeclareMathOperator{\Ecal}{\mathcal{E}}

\DeclareMathOperator{\Vcal}{\mathcal{V}}

\DeclareMathOperator{\Gfrak}{\mathfrak{G}}

\DeclareMathOperator{\TV}{TV}

\title{Augmentation Algorithms for Integer Programs with Total Variation-like Regularization}
\date{\today}

\author*[1]{\fnm{Dominic} \sur{Yang}}\email{dominic.yang@anl.gov}

\author[1]{\fnm{Sven} \sur{Leyffer}}\email{leyffer@anl.gov}

\author[2]{\fnm{Miles} \sur{Bakenhus}}\email{mbakenhus@hawk.iit.edu}

\affil*[1]{\orgdiv{MCS Division}, \orgname{Argonne National Laboratory}, \country{USA}}

\affil[2]{\orgdiv{Department of Applied Mathematics}, \orgname{Illinois Institute of Technology}, \country{USA}}

\begin{document}

\abstract{We address a class of integer optimization programs with a total variation-like regularizer and convex, separable constraints on a graph. Our approach makes use of the Graver basis, an optimality certificate for integer programs, which we characterize as corresponding to the collection of induced connected subgraphs of our graph. We demonstrate how to use this basis to craft an exact global optimization algorithm for the unconstrained problem recovering a method first shown by Kolmogorov and Shioura in 2009. %%Avoidd citations in abstract \cite{kolmogorov2009new}. 
    We then address the problem with an additional budget constraint with a randomized heuristic algorithm that samples improving moves from the Graver basis in a randomized variant of the simplex algorithm. Through comprehensive experiments, we demonstrate that this randomized algorithm is competitive with and often outperforms state-of-the-art integer program solvers.}



\maketitle

% \Miles{
% Edges and standard basis vectors have similar notation:}
% \begin{quote}
%     $e\in E$ is an edge, but $e_i$ is the $i^{\text{th}}$ standard basis vector. We could do something like using bold to represent matrices and vectors: $\boldsymbol{A}$ and $\boldsymbol{x}$, but that might give us something like $\boldsymbol{e}_e$ which is also confusing. Alternatively we can use something like $\sigma\in E$ to represent edges, but that might not be good standard notation.
% \end{quote}
% \Miles{
% Notation Decisions:
% \begin{itemize}
%     \item Change all graph/hypergraph notation to $\Gcal = (\Vcal, \Ecal)$ and $\Hcal = (\Vcal, \Ecal)$
%     \item Hyperedges: $E\in \Ecal$, undirected edges $uv \in \Ecal$, directed edges $\overrightarrow{uv},\overrightarrow{vu} \in \Ecal$ 
%     \item Constraint matrix $A$, Oriented Incidence matrix $B$
%     \item Graver basis: $\Gfrak_A$
%     \item weights $w_E$, $w_{uv}$ instead of $d_{uv}$
%     \item global def: $J(x) = F(x) + G(x)$, $F(x) = \sum_v F_v(x)$, $G(x)= \sum_{E}G_E(x)$ or $G(x) = \sum_{uv} G_{uv}(x)$
%     \item $\hat{x}$ for all estimated solutions
%     \item use $\Xcal = \{0,\dots,Q\}^{|\Vcal|}$
%     \item use $\Zfrak$ instead of $\mathcal{U}$ for sets of augmenting moves, and $z$ instead of $u$ for augmenting moves
% \end{itemize}
% }

% \tableofcontents

\section{Introduction}

% \Dominic{
% \begin{enumerate}
%     \item Immediately introduce Fully nonlinear version of problem
%     \item Introduce total variation + applications
%     \begin{enumerate}
%         \item Total variation Denoising
%         \item Total variation regularization in topology optimization etc. (Sven could write something here)
%         \item Trust region subproblem
%     \end{enumerate}
%     \item Other possible applications of fully nonlinear form (Potts model)
% \end{enumerate}
% Come up with consistent notation,
% For Miles, compare Hypergraph total variation with standard total variation, try to come up with concrete application of problem
% }

% \Miles{This section needs more background information discussing our motivation for studying this problem and references other work on similar problems. What are the applications (PDE constrained optimization, Image denoising)? What are the standard methods for solving these (branch-and-bound)? Is there existing literature that directly relates to our methods (MRFs, min-cut, etc.).}

% \Miles{Introduce/Clean-up full nonlinear problem def.}

We address combinatorial optimization problems of the following form:
\begin{subequations}
\begin{align}
    \min_{x\in \Z^n} \qquad& J(x) := \sum_{v \in \Vcal}F_v(x_v) + \sum_{uv \in \Ecal}G_{uv}(x_u -x_v) \\
    \text{s.t.} \qquad& \sum_{v \in \Vcal} H_v(x_v) \le \Delta, \label{eq:budget-constraint}\\ 
    &x \in \{0,\dots,Q\}^{\Vcal},
\end{align} \label{eq:General-Problem}%
\end{subequations}
where we have an associated directed graph $\Gcal = (\Vcal, \Ecal)$;  $F_v,$ $G_{uv},$ and $H_v$ are univariate functions that are assumed to be convex; and $Q \in \mathbb{Z}_{\ge 0}$ and $\Delta \ge 0$ are given parameters. We define
% Furthermore, we assume there is an associated labeled graph $\Gcal = (\Vcal,\Ecal)$ with vertices $\Vcal$ of order $|\Vcal| = n$ indexing the variables $x = (x_v)_{v\in \Vcal}$, and edges corresponding to the structure of $G(x)$  such that 
\begin{equation}
F(x) = \sum_{v\in \Vcal}F_v(x), \quad G(x) = \sum_{uv \in \Ecal}G_{uv}(x_u - x_v), \quad \text{and} \quad  H(x) = \sum_v H_v(x_v).
\label{eq:obj-graph-structure}
\end{equation}
We refer to the constraint $H$ as the budget constraint for \eqref{eq:General-Problem}.
Despite its simple form, 
%\eqref{eq:General-Problem} 
this is a challenging combinatorial optimization problem with a number of important applications, including image denoising \cite{rudin1992nonlinear} and topology optimization \cite{manns2024discrete}.

In the absence of constraint \eqref{eq:budget-constraint}, this problem was studied in the late 1990s and 2000s extensively in the computer vision and pattern recognition community, and most algorithms addressing this problem use a minimum cut approach. 
In 1998 \cite{ishikawa1998segmentation} (and later refined in 2003 \cite{ishikawa2003exact}), Ishikawa produced one of the first exact algorithms for solving the unconstrained problem even when $F_v$ are general functions. 
We note in particular the method of Kolmogorov and Shioura in \cite{kolmogorov2009new} that exploits the fact that $J$ is an $L^{\natural}$-convex function for which there are known polynomial-time algorithms (see \cite{murota2004optimality,murota2004steepest} for definitions and algorithms). 
For this class of functions, it is known that if at a given point $x$, $J(x) \le J(x \pm \chi_X)$ for all choices of $X \subset \Vcal$ where $x \pm \chi_X$ is feasible, $x$ is globally optimal. 
This leads to a simple augmentation routine shown in Algorithm \ref{alg:NL-Augmentation} that repeatedly solves binary optimization problems to determine the subsets $X$ to increment or decrement on, each of which can be found by solving an associated minimum cut problem.

\begin{algorithm}
\caption{Augmentation Algorithm for Unconstrained Version of \eqref{eq:General-Problem} (from \cite{kolmogorov2009new})}\label{alg:NL-Augmentation}
\DontPrintSemicolon
\KwData{$x^{(0)} \in \{0,1,\ldots,Q\}^{\Vcal}$}
$k \gets 0$\;
\While{True}
{    
    $\Delta x^{\text{up}} \gets \text{argmin}_{\Delta x \in \{0,1\}^{\Vcal}, x^{(k)}+\Delta x \le Q} J(x^{(k)} + \Delta x)$\;
    %Solve \eqref{eq:Alt-Up-Augment} with $\hat{x} = x^{(k)}$ for $\Delta x^{\text{up}}$\;
    %$\alpha_k \gets \text{argmin}_{\alpha \in \Z_{\ge 0}} J(x^{(k)} + \alpha \Delta x^{\text{up}})$\; 
    $x^{(k+1)} \gets x^{(k)} + \Delta x^{\text{up}}$\;
    $\Delta x^{\text{down}} \gets \text{argmin}_{\Delta x \in \{0,1\}^{\Vcal}, x^{(k+1)}-\Delta x \ge 0} J(x^{(k+1)} - \Delta x)$\;
    %Solve \eqref{eq:Alt-Down-Augment} with $\hat{x} = x^{(k+1)}$ for $\Delta x^{\text{down}}$\;
    %$\alpha_{k+1} \gets \text{argmin}_{\alpha \in\Z_{\ge 0}} J(x^{(k+1)} + \alpha \Delta x^{\text{down}})$\;
    $x^{(k+2)} \gets x^{(k+1)} - \Delta x^{\text{down}}$\;
    \If(\Comment*[f]{Neither Augmentation Changed $x$}){$x^{(k+2)} = x^{(k)}$}
    {
        \Return $x^{(k)}$ \Comment{Point is Optimal}
    }
    $k \gets k + 2$\;
}
\end{algorithm}

The methods we develop in this paper bear similarity to this approach, but we ground our reasoning in the context of Graver basis augmentation methods. 
The Graver basis, first introduced in \cite{graver1975foundations}, is a collection of vectors associated with a given integer matrix $A$ that most importantly serve as an optimality certificate for any integer program with a convex, separable objective that uses $A$ as its constraint matrix (see \cite{de2012algebraic} for a comprehensive introduction). In Theorem \ref{thm:graver-basis} we establish the structure of the Graver basis for matrix $A_{TV}$ associated with the constraints $\{x_u - x_v = a_{uv}\}_{uv \in \Ecal}$ establishing that $\{\pm \chi_S : S \subset \Vcal, S \text{ induces a connected subgraph}\}$ is an optimality certificate for \eqref{eq:General-Problem}.
Hence, instead of optimizing over all subsets of $\{0,1\}^{\Vcal}$, we can modify each of the subproblem solves in Algorithm \ref{alg:NL-Augmentation} to  consider only subsets that induce connected subgraphs. 

This alternate basis presents an opportunity for the development of randomized heuristics where we sample these connected subgraphs.
In this sense we are now modifying Algorithm \ref{alg:NL-Augmentation} to replace the exact subproblem solves with randomized approximate solves given by sampling the Graver basis.
Since $A_{TV}$ is well known to be totally unimodular, edges of any polytope where $A_{TV}$ is the constraint matrix are necessarily integer multiples of Graver basis elements \cite{thomas1995geometric}.
Therefore, algorithms that traverse the edges of the underlying polytope, in particular primal simplex methods, effectively perform Graver basis augmentation algorithms.
In our work we establish how to randomize the operation of a primal simplex method to sample improving Graver augmentations as well as adapt the method to ensure feasibility according to the constraint \eqref{eq:budget-constraint}. 

\paragraph{Related Work.}
Several other polynomial-time algorithms exist for handling the unconstrained version of \eqref{eq:General-Problem}. 
In \cite{ahuja2004cut} and \cite{hochbaum2001efficient}, the authors present  minimum cut algorithms for solving the problem.  Boykov et al. in \cite{boykov2002fast} and \cite{boykov2004experimental} address the more general situation where $G_{uv}$ is an arbitrary function of two variables and provide approximation guarantees for their algorithm. The authors in \cite{bioucas2007phase} address specifically the case without the first term $F$. 

The full problem with the budgetary constraint included has seen comparatively less study in the literature. It is known to contain NP-hard problems including finding a minimum $s-t$ cut of a directed graph with a bound on the number of vertices on one side (shown in \cite{armon2006multicriteria}). Recently, a few papers by Manns and Severitt have focused  on the case where $F$ is linear, $G$ is the total variation, and $H$ restricts the feasible space to be in a 1-norm ball about a given point. In \cite{severitt2023efficient}, they study a one-dimensional variant of the total variation, i.e., where the underlying graph in \eqref{eq:General-Problem} is a path graph and they give a polynomial algorithm for this setting. In \cite{manns2024discrete}, they study the two-dimensional version of the anisotropic total variation where the underlying graph is an $M\times N$ grid and give primal heuristics, branching rules and methods for producing cutting planes for this problem.

We are broadly interested in problems where $G_{uv}$ is the absolute value, in which case $G$ is understood to be the total variation (TV) of $x$ with respect to the underlying graph $\Gcal$ first introduced as a regularizer in \cite{rudin1992nonlinear}. In its original context as an image denoiser, it was notable  for its ability to smooth the solution while still preserving sharp curves in the image. A few papers \cite{zalesky2002network,chambolle2005total} also specifically address the unconstrained version of \eqref{eq:General-Problem} where $G$ is the anisotropic total variation of a graph. Recently, it has seen use in the context of topology optimization \cite{manns2024discrete,severitt2023efficient} in order to limit speckle patterns in produced solutions. We now review two applications of \eqref{eq:General-Problem} where the total variation is explicitly used as a regularizer which we will refer to later in the paper.

% \Dominic{Add background on total variation, some basic discussion of application areas}
\paragraph{Total Variation Denoising.}

The classic setting for problems of the form \eqref{eq:General-Problem} is in the context of image denoising. In this setting we have an observed image $\hat{x} \in \mathbb{R}^{M\times N}$ that has been corrupted by  noise, and we want to produce an image that approximates the original image. To this end, we set up the following optimization problem:
\begin{subequations}
\begin{align}
    \min_{x} \qquad & \sum_{i=1}^{M}\sum_{j=1}^N (x_{ij} - \hat{x}_{ij})^2 + \alpha\left(\sum_{i=1}^{M}\sum_{j=1}^{N-1} |x_{i,j+1} - x_{ij}| + \sum_{i=1}^{M-1}\sum_{j=1}^N |x_{i+1,j} - x_{ij}|\right) \\
    \text{s.t.} \qquad & x_{ij} \in \{0,1,\ldots,Q\}^{M \times N}.
\end{align} \label{eq:image-reconstruction}%
\end{subequations}
The first term in the objective is a fidelity term and ensures that the produced solution remains close to the noisy image;  the latter two sums constitute the total variation. In this expression, $\alpha$ dictates the relative weighting of data fidelity and total variation. 

A simple extension is the addition of a bound over the sum of the $x$ variables:
\begin{equation}
    \sum_{i=1}^M\sum_{j=1}^N x_{ij} \le \Delta . \label{eq:image-budget-constraint}\\
\end{equation}
This constraint may seem artificial in the context of image denoising; but in topology optimization contexts where we are often interested in recreating the solutions in a physical domain, this constraint imposes a budgetary constraint on the amount of material allowed. For this reason we call \eqref{eq:image-reconstruction} with the additional constraint \eqref{eq:image-budget-constraint} the \textbf{image reconstruction problem}.

\paragraph{Trust-Region--Constrained Total Variation Problem.}
When $F$ is a general nonlinear function, we may produce approximate solutions to \eqref{eq:General-Problem} by linearizing $F$ about a point $\hat{x}$ as $F(x) \approx \sum_v (F(\hat{x}) + \partial F/\partial x_v(\hat{x})(x- \hat{x}))$.
These problems emerge as subproblems in a broader optimization scheme, where we take this linear model and pose the optimization problem over a restriction of the original feasible space to a smaller region about $\hat{x}$ where the approximation of $F$ is good. This smaller region,  called a trust region, is typically modeled as a ball in a given norm of radius $\Delta$ where $\Delta$ is chosen appropriately. We repeatedly solve this problem; and depending on the quality of solutions produced, we accept the solution and produce a new model about the next point or reject the solution and reduce $\Delta$ to give a better model. This leads to following subproblem where $c_v = \partial F/\partial x_v(\hat{x})$ and we choose the 1-norm:
\begin{subequations}
\begin{align}
    \min_{x} \qquad & \sum_{v\in \Vcal} c_vx_v + \alpha\sum_{uv \in \Ecal} |x_u -x_v|  \\
    \text{s.t.} \qquad &\|x-\hat{x}\|_1 \le \Delta \\
    &x \in \{0,1,\ldots,Q\}^{\Vcal}.
\end{align}\label{eq:trust-region-subproblem}%
\end{subequations}
%This problem has been studied in \cite{severitt2023efficient,manns2024discrete} where the problem is conjectured to be NP-complete.
We call this problem the \textbf{trust-region--constrained subproblem} and note that it often appears in the context of topology optimization where the total variation regularization promotes manufacturability constraints.
It arises as a subproblem in the sequential linearization approach to  solve integer optimal control problems regularized with a total variation penalty; see \cite{leyffer2022sequential}. 
In the case of a one-dimensional domain of the control function, it  can be shown that the iterates produced by a trust-region algorithm converge to points that satisfy certain first-order stationarity conditions for local optimality.

\paragraph{Terminology and Notation.}
We introduce some terminology and notation that will appear throughout the  paper. For every optimization problem we work with, we will always have an associated directed graph $\Gcal = (\Vcal, \Ecal)$, where $\Vcal$ is the set of vertices and $\Ecal \subset \Vcal \times \Vcal$ is the set of edges. We represent an edge that leaves vertex $u$ and arrives at vertex $v$ with the notation $uv$.  In certain cases (in particular for total variation problems), there may not be a natural orientation on the graph, and any choice of orientation will suffice.

A \textit{path} in $\Gcal$ is a sequence of vertices $v_1, \ldots, v_k$ such that $v_iv_{i+1} \in \Ecal$ for $i=1,\ldots,k-1$. A \textit{subgraph} of $\Gcal$ is a graph $\Gcal' = (\Vcal', \Ecal')$ where $\Vcal' \subset \Vcal$ and $\Ecal' \subset \Ecal \cap (\Vcal' \times \Vcal')$. A subgraph is \textit{strongly connected} if for any two vertices $u, v\in \Vcal'$ there is a path from $u$ to $v$ and from $v$ to $u$;  it is \textit{weakly connected} if, when ignoring directions, there is a path between any two vertices. Unless otherwise specified, when we refer to a graph as connected, we will mean it in the weak sense. Given a subset of vertices $S \subset V$, $S$ induces a subgraph $\Gcal' = (S, \Ecal')$, where $\Ecal'$ comprises all edges in $\Ecal$ that are between vertices in $S$, that is, $\Ecal' = (S \times S) \cap \Ecal$;  we call any such subgraph an \textit{induced subgraph}. We denote the set of outgoing edges from $S$ as $\delta^+(S) = \{uv \in \Ecal: u \in S\}$ and the set of incoming edges to $S$ as $\delta^-(S) = \{uv \in \Ecal: v \in S\}$. The union of both sets is denoted $\delta(S) = \delta^+(S) \cup \delta^-(S)$.
Given a vertex $u\in \Vcal$, we let $e_u$ denote a vector of length $|\Vcal|$ that is 1 at the index corresponding to $u$ and 0 elsewhere. Similarly, given an edge $vw \in \Ecal$, $e_{vw}$ is a vector of length $|\Ecal|$ that is 1 at the index for $vw$ and 0 elsewhere. More generally, for $S \subset \Vcal$, $\chi_S$ is a vector of length $|\Vcal|$ that is 1 on $S$ and 0 elsewhere; and for $T \subset \Ecal$, $\chi_T$ is a vector of length $|\Ecal|$ that is 1 on $T$ and 0 elsewhere.

\paragraph{Outline of Contributions.}
In this paper we analyze the underlying structure of optimization problems with total variation-like regularization and exploit this structure to devise efficient algorithms for these problems and variants of these problems. 
In Section \ref{sec:structure} we characterize in Theorem \ref{thm:graver-basis} the structure of the Graver basis for a linearization of the unconstrained version of \eqref{eq:General-Problem} as corresponding to the collection of connected subgraphs of $\Gcal$. This gives an alternate proof of optimality for Algorithm \ref{alg:NL-Augmentation}. We also discuss how to set up the augmentation subproblem solves in the algorithm.
In Section \ref{sec:trust-region} we describe how to sample improving moves from the Graver basis by randomizing the simplex method  for these subproblems. We demonstrate how this can be used to construct a randomized algorithm for approximately solving \eqref{eq:General-Problem}, and we explore various properties of this randomized algorithm.
In Section \ref{sec:experiments} we establish a collection of experiments on two problems that demonstrate an improvement in both objective and solve time when compared with standard integer program solvers.

\section{Structure in Total Variation-Regularized Problems}\label{sec:structure}

We begin our study by first establishing properties of our main problem \eqref{eq:General-Problem} in the absence of the budget constraints $H(x) \le \Delta$. We start by giving a review of the Graver basis and discuss its use as an optimality certificate for integer programs. We then establish the Graver basis associated with our unconstrained problem as corresponding to induced connected subgraphs of $\Gcal$ and touch on how its structure impacts augmentation algorithms on \eqref{eq:General-Problem}. In spite of the exponential size of the Graver basis, we demonstrate the augmentation subproblems can be solved efficiently by exploiting total unimodularity inherent in the problem.  We also discuss a specific linearization when the function $G$ corresponds directly to the total variation which we will make use of when we study \eqref{eq:General-Problem} in Section \ref{sec:trust-region}.

\subsection{Graver Bases and the Basic Augmentation Algorithm}

%Standard approaches to solving this problem exactly typically involve a branch-and-bound scheme where the feasibility space is recursively partitioned to determine possible instantiations of the variables $x$ (see \cite{wolsey2020integer} for more information on these approaches). 
% Integral to this approach is the solution of linear relaxations of the problem which provide lower bounds by which the various subproblems in the recursive search may be pruned. Owing to this fact, the polyhedral representation of the linear relaxation as tighter representations of the problem will provide better bounds which have the potential to significantly accelerate the search process. Solutions produced in this process exist as extreme points of the polyhedra of linear relaxations of these problems. 

% Consider an integer program in standard form
% \begin{equation}
% \begin{aligned}
%     \min_{x} &\qquad c^Tx \\
%     \text{s.t.} &\qquad Ax = b \\
%         & x \in \Z_{\ge 0}^n
% \end{aligned} \tag{IP} \label{eq:IP}
% \end{equation}
% where $c \in \R^n, A \in \Z^{m\times n}$, and $b \in \R^m$. 

% \Miles{To be consistent with the Graver basis notation, do we want to use $\Ufrak$ instead of $\mathcal{U}$ for the augmenting moves? Also we use $u$ to denote a vertex in other places, so it might make sense to use $\gfrak\in \Gfrak{_A}$ and $\ufrak\in \Ufrak$ to represent Graver/Augmenting moves}

The central tool that we employ for optimizing \eqref{eq:General-Problem} is the augmentation algorithm.
Augmentation approaches differ from the standard branch-and-bound approach to solving integer programs in that solutions are produced by traveling through the integral lattice within the feasible polytope. The basic structure of such an approach is as follows: Given an initial feasible point $x_0$ and a collection of augmenting moves $\mathcal{U}$, determine whether there exists a move $u \in \mathcal{U}$ that improves the objective value (i.e., $J(x_0 + u) < J(x_0)$) and maintains feasibility. If such a move exists, we take it and set $x_1 = x_0 + u$ and repeat; otherwise, we stop. Note that Algorithm \ref{alg:NL-Augmentation} is an augmentation algorithm taking $\mathcal{U}$ to be all subsets of $\{0,1\}^{\Vcal}$.
% Many classical combinatorial algorithms take the form of augmentation algorithms, e.g., augmenting paths for maximum flow \cite{ford1956maximal}, alternating paths for perfect matchings \cite{edmonds1965paths}, and cycles for minimum cost flow \cite{kantorovich1942translocation}.

In our study we take $\mathcal{U}$ to be the Graver basis (introduced in \cite{graver1975foundations}; see \cite{de2012algebraic} for a comprehensive overview) of the constraint matrix.  To introduce the Graver basis, we first introduce a special partial ordering on $\Z^n$. 
That is, we say that for $x,y \in \Z^n$, $x \sqsubseteq y$ if and only if $x_iy_i \ge 0$ and $|x_i| \le |y_i|$ for all $i=1,\ldots,n$. 
The first condition requires that for two points to be comparable, they need to lie in the same orthant; and the second condition can be understood as requiring $x$ to be closer to the origin with respect to all coordinates. 
If $x \sqsubseteq y$, we say that $y$ majorizes $x$. 
We also denote the integer kernel of a matrix $A \in \mathbb{Z}^{m\times n}$ as $\text{ker}_{\mathbb{Z}}(A) := \text{ker}(A) \cap \mathbb{Z}^n$. We define the Graver basis as follows.

\begin{definition}
    Given a matrix $A \in \Z^{m\times n}$, the \textbf{Graver basis} of $A$, $\Gfrak_A$ is the set of $\sqsubseteq$-minimal elements of $\ker_{\Z}(A) \setminus \{0\}$.
\end{definition}

The elements of the Graver basis are also known as ``primitive'' elements because they cannot be further decomposed into smaller Graver basis elements.
Importantly, the Graver basis is an optimality certificate for any choice of convex objectives $F_v$ and $G_{uv}$ in \eqref{eq:sub}. 
By this we mean that if we have a suboptimal point, there always exists a Graver basis element that can be added to the point to improve the objective while maintaining feasibility. The absence of such a move is proof of global optimality. 

Graver basis methods have been less-used approaches in integer programming primarily because of the generally intractable computational burden needed to fully compute a Graver basis; however, in certain domains these methods have produced efficient algorithms. 
In particular, Graver basis methods have given efficient algorithms for $N$-fold programming \cite{de2008n,hemmecke2013n}, certain classes of 2-stage stochastic problems \cite{hemmecke2003decomposition, klein2020complexity}, and certain quadratic optimization problems \cite{lee2012quadratic}.

We now set up the \textbf{Graver augmentation subproblem} for minimizing $J(x)$ subject to $Ax = b$ and $x \in \{0,1,\ldots,Q\}^n$ about some feasible point $\hat{x}$. This is detailed in \eqref{eq:Aug-Subproblem}:
\begin{equation}
\begin{aligned}
    \min_{g, \alpha} \qquad & J(\hat{x} + \alpha g) \\
    \text{s.t.} \qquad & \hat{x} + \alpha g \in \{0,1,\ldots,Q\}^n,\\
                & g \in \mathfrak{G}_A, \alpha \in \mathbb{Z}_{\ge 0}.
\end{aligned} \label{eq:Aug-Subproblem}
\end{equation}
A global optimization routine can be devised by repeatedly solving \eqref{eq:Aug-Subproblem} and augmenting with the solution until no more improving moves can be found. This routine is presented in Algorithm \ref{alg:graver-augmentation}. Provided that the feasible space is bounded and therefore finite, Algorithm \ref{alg:graver-augmentation} produces a sequence of feasible iterates with strictly decreasing objective terminating in a finite number of steps at the global optimum. If $n$ is the number of variables and $J_{\Delta}$ is the maximal difference in objective between two feasible points, at most $O(n\log J_{\Delta})$ iterations are needed \cite{de2012algebraic}.

\begin{algorithm}
\caption{Conceptual Graver Basis Augmentation Algorithm} \label{alg:graver-augmentation}
\KwData{Feasible $x^{(0)}$, Graver basis $\Gfrak_A$}
$k \gets 0$\;
\While{True}
{    $(g, \alpha) \gets $ solution of Graver Augmentation Problem \eqref{eq:Aug-Subproblem} with $\hat{x} = x^{(k)}$\;
    \eIf{$J(x^{(k)} + \alpha g) = J(x^{(k)})$} {
        Return $x^{(k)}$ \Comment{Move did not improve objective}
    }{
        $x^{(k+1)} \gets x^{(k)} + \alpha g$\;
    }
    $k \gets k + 1$\;
}
\end{algorithm}

Unfortunately, the conceptual algorithm~\ref{alg:graver-augmentation} is largely of theoretical value because it requires a priori knowledge of the Graver basis, which in the worst case may grow exponentially for many problems. To derive a computationally practical implementation, we next analyze the Graver structure of \eqref{eq:sub} and then show how this leads to a practical algorithm.

\subsection{Polyhedral Structure of the Constraints}
To analyze the properties of the unconstrained version of problem \eqref{eq:General-Problem}, we first perform a variable substitution, introducing a variable $a_{uv} = x_u - x_v$ for each edge $uv \in \Ecal$. In this setting our problem becomes
\begin{subequations}
\begin{align}
    \min_{x,a} \qquad &\sum_{v \in \Vcal} F_v(x_v) + \sum_{uv \in \Ecal} G_{uv}(a_{uv}) \\
    \text{s.t.} \qquad & x_u - x_v = a_{uv}, \quad uv \in \Ecal \\
                & x \in \{0,1,\ldots,Q\}^{\Vcal}, a \in \mathbb{Z}^{\Ecal} .
\end{align}\label{eq:sub}%
\end{subequations}
We denote the polytope of the linear relaxation of \eqref{eq:sub} by $\mathcal{P}_{A_{TV}}$.
Our objective function is now separable in the variables $x$ and $a$; and the constraint matrix, which we denote $A_{TV}$, takes the form
\begin{equation}
    A_{TV} = \begin{bmatrix} B & -I \end{bmatrix},
\end{equation}
where $I$ is the $|\Ecal| \times |\Ecal|$ identity matrix and $B$ is the transpose of an oriented incidence matrix for the underlying graph $\Gcal$, namely, the $|\Ecal| \times |\Vcal|$ matrix where $B_{uv,w}$ is $1$ if $u=w$, $-1$ if $v = w$, and $0$ otherwise.
It is well known that the oriented incidence matrix is totally unimodular; and since total unimodularity is preserved under concatenation with an identity matrix, $A_{TV}$ is totally unimodular. We then have the following remark following standard integer programming theory.
\begin{remark}\label{rmk:verts-are-integral}
    The vertices of $\mathcal{P}_{A_{TV}}$ are integral; and for $(x, a) \in \mathcal{P}_{A_{TV}}$, $x \in \{0,Q\}^{\Vcal}$.
\end{remark}
%\textcolor{red}{Do I want to include a proof: ugh}. 

The total unimodularity also has consequences in terms of the edges of the underlying polytope. From \cite{sturmfels1997variation}, every edge (understood as a vector between two vertices) is a multiple of a universal Gr\"obner basis element (see \cite{sturmfels1997variation} for a definition). Since the Graver basis contains the universal Gr\"obner basis, we have the following statement.
\begin{remark}\label{rmk:edges-are-graver}
    Suppose $u,v$ are two adjacent vertices of the integer polytope given by constraints $Ax = b$ and $x \ge 0$, where $A$ is unimodular. Then, for any choice of $b \in \Z^n$, $u-v$ is an integral multiple of some element in the Graver basis of $A$.
\end{remark}

\subsection{Graver Structure of Augmented Oriented Incidence Matrix}
We expound upon the Graver basis structure of \eqref{eq:sub}, which will establish the fundamental units on which we can derive an efficient augmentation algorithm. In this context the Graver basis takes a particular form related to the connectivity of the underlying graph. That is, the fundamental units of the kernel of $A_{TV}$ are vectors that change $x_v$ uniformly on a connected subset of vertices $S$, as described in the following theorem.

% \Dominic{I am going to suggest we move notation to a dedicated section} \textcolor{red}{We introduce some notation to facilitate this characterization. We will write $e_v$ for $v \in \Vcal$ to denote a vector of length $|\Vcal|$ which is zero everywhere except for the coordinate corresponding to $v$ where it is 1. Analogously, we will write $e_{uv}$ for $uv \in \Ecal$ to denote a vector of length $|\Ecal|$ which is zero everywhere except for the coordinate corresponding to $uv$ where it is 1. Given a set of vertices $S \subset \Vcal$, we write $e_S = \sum_{v \in \Scal} e_v$.}

\begin{theorem}\label{thm:graver-basis}
    The Graver basis of $A_{TV}$, denoted by $\Gfrak_{A_{TV}}$,  comprises precisely the points $(x,a)$, where $x = \chi_S$ for $S \subset \Vcal$, where $S$ induces a connected subgraph of $\Gcal$, and $a = \chi_{\delta^+(S)} - \chi_{\delta^-(S)}$ as well as the negatives of such points.
\end{theorem}

A visualization of the Graver basis for a $3\times 2$ grid graph is presented in Figure \ref{fig:3x2-graver}. In general, it is not practical to construct the Graver basis, which can grow exponentially with size, as illustrated in Table~\ref{tab:connected-graph-count}. 

To prove Theorem \ref{thm:graver-basis}, we first remark that the integer kernel of $A_{TV}$ is easy to characterize: for $x \in \mathbb{Z}^{\Vcal}$, the variables $a_{uv}$ are exactly determined as $a_{uv} = x_u - x_v$. From this, we simply need to determine the $\sqsubseteq$-minimal points $(x,a)$ of $\text{ker}_{\mathbb{Z}}(A_{TV}) \setminus \{0\}$. The following lemma aids in this task.

\begin{lemma} \label{lem:TV-majorizer}
    Let $(x, a) \in \ker_{\Z}(A_{TV})$. Let $k = \max_{v:x_v > 0} x_v$ and $S$ be the vertices of any connected component of the subgraph induced by $\{v \in \Vcal : x_v = k\}$. Then $(\chi_S, \chi_{\delta^+(S)} - \chi_{\delta^-(S)}) \sqsubseteq (x, a)$. An analogous result holds for $k = \min_{v: x_v < 0} x_v$ but with $(-\chi_S, -\chi_{\delta^+(S)} + \chi_{\delta^-(S))})$ instead. 
\end{lemma}
\begin{proof}
    We prove the first part only since the second can be proven in the same way with appropriate sign changes. Let $(x,a)$, $k$, and $S$ be as stated in the lemma. It is immediate that $\chi_S \sqsubseteq x$, so we need only check that $\chi_{\delta^+(S)} - \chi_{\delta^-(S)} \sqsubseteq a$. Note that since $k$ is maximal, we have for each $uv \in \delta^+(S)$, $a_{uv} = x_u - x_v \ge 1$ since $x_u = k$ and $x_v < k$ (as $u \in S$ and $v \notin S$). Similarly, for $uv \in \delta^-(S)$, where $u \notin S$ and $v\in S$, $a_{uv} = x_u - x_v \le -1$. Therefore, $\chi_{\delta^+(S)} - \chi_{\delta^-(S)} \sqsubseteq a$, and the statement is proven.
\end{proof}

Using this lemma, we prove Theorem~\ref{thm:graver-basis}.

\begin{proof}[Proof of Theorem \ref{thm:graver-basis}] 
    Let $(x, a) \in \ker_{\Z}(A) \setminus \{0\}$. From Lemma \ref{lem:TV-majorizer} we can construct an element $g_S = \pm (\chi_S, \chi_{\delta^+(S)} - \chi_{\delta^-(S)})$, where $g_S \sqsubseteq (x, a)$. If $g_S \ne (x, a)$, then $(x,a)$ cannot be in the Graver basis because $(x,a)$ is not $\sqsubseteq$-minimal. Hence, any $(x,a)$ that is not of the form $\pm (\chi_S, \chi_{\delta^+(S)} - \chi_{\delta^-(S)})$ cannot be a Graver basis element.

    Now we verify that these moves are $\sqsubseteq$-minimal. Suppose for given $g_S = (\chi_S, \chi_{\delta^+(S)} - \chi_{\delta^-(S)})$ we can find $g_T = (\chi_T, \chi_{\delta^+(T)} - \chi_{\delta^-(T)})$ with $g_T \sqsubseteq g_S$ and $g_T \ne g_S$. Then we must have $T \subset S$ where containment is strict. Since $S$ is a weakly connected component, there exists an edge $uv$ between $T$ and $S \setminus T$. Both $u$ and $v$ are in $S$, so $g_S$ is 0 on $uv$; but since only one is in $T$, $g_T$ must be nonzero on $uv$. This contradicts $g_T \sqsubseteq g_S$, and hence $g_S$ must be in the Graver basis.
\end{proof}

\begin{figure}
    \centering
    \includegraphics[width=\linewidth]{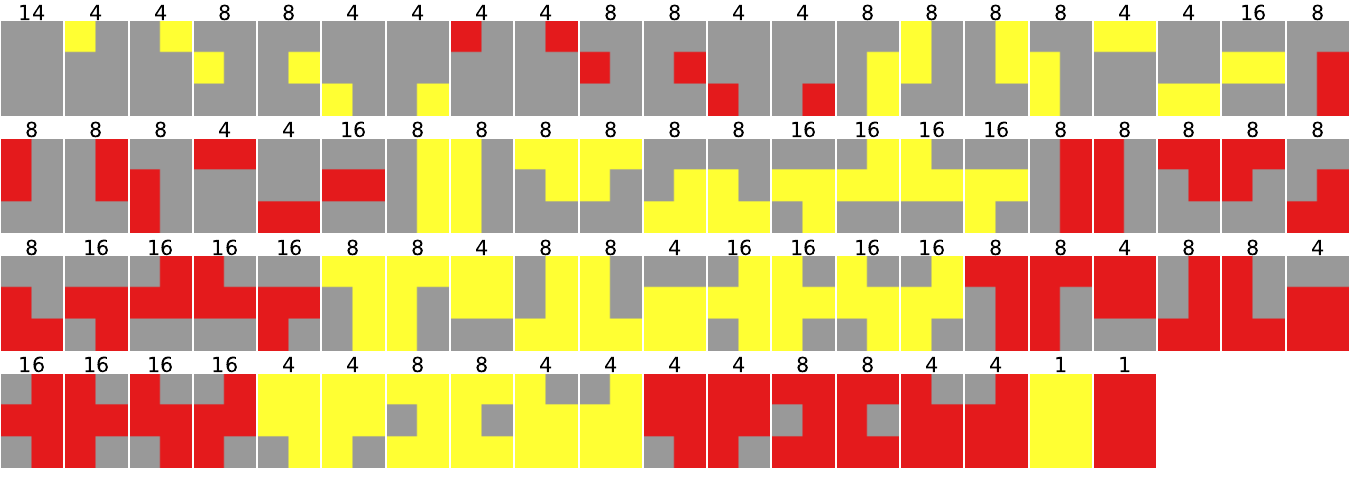}
    \caption{Projection of Graver basis $\Gfrak_{A'_{TV}}$ for a $3\times 2$ grid onto the $x$ variables. The colored pixels indicate vertices $S$ inducing a connected subgraph where $x_v = 1$ if $v$ is yellow, $x_v = -1$ if $v$ is red and 0 otherwise. The number above a given move $x$ corresponds to how many moves in $\Gfrak_{A'_{TV}}$ project onto $x$. The projection of $\Gfrak_{A_{TV}}$ is given by taking one of each element from the second element onward.}
    \label{fig:3x2-graver}
\end{figure}
Given that the Graver basis serves as an optimality certificate, we have the following corollaries characterizing optimality for the unconstrained version of \eqref{eq:General-Problem}.
\begin{corollary}\label{cor:graver-optimality}
    Suppose $x \in \{0,1,\ldots,Q\}^{\Vcal}$ satisfies $J(x) \le J(x \pm \chi_S)$, where $S$ induces a connected subgraph of $\Gcal$ and $x \pm \chi_S$ is feasible. Then $x$ is a globally optimal solution for the unconstrained version of \eqref{eq:General-Problem}.
\end{corollary}
This also immediately gives an alternate proof that Algorithm \ref{alg:NL-Augmentation} converges to the globally optimal solution.
\begin{theorem}[first shown in \cite{kolmogorov2009new}]\label{thm:global-convergence}
    Algorithm 1 converges to the globally optimal solution of the unconstrained version of \eqref{eq:General-Problem} in a finite number of iterations.
\end{theorem}
\begin{proof}
    In Algorithm \ref{alg:NL-Augmentation} each augmentation subproblem optimizes over all moves $\pm \chi_X$, where $X$ is any subset of $\Vcal$. This is a strict superset of the collection of Graver moves and therefore an optimality certificate.
\end{proof}

We also have a natural corollary regarding Algorithm \ref{alg:NL-Augmentation}.
\begin{corollary}
    If we modify Algorithm \ref{alg:NL-Augmentation} so that each subproblem searches only over vectors $\chi_S$, where $S$ induces a connected subgraph, it converges to the global optimum. 
\end{corollary}
This corollary suggests that we could potentially speed up Algorithm \ref{alg:NL-Augmentation} by replacing the subproblems with searches over these reduced sets. However, the subproblems  are already efficiently solvable with minimum cut routines (see \cite{kolmogorov2009new}), and imposing a connectivity requirement over moves may actually make the subproblems harder to solve. Furthermore, for Algorithm \ref{alg:NL-Augmentation} we have a good bound on the number of iterations, $2Q+2$ \cite{kolmogorov2009new}, whereas for this amended algorithm we have a weaker bound $O(n J_\Delta)$ \cite{de2012algebraic}, where $n = |\Vcal| + |\Ecal|$ and $J_\Delta$ is the range of function $J$.

From Remark \ref{rmk:edges-are-graver} we also have an immediate corollary that characterizes the edges of the polytope of the linear relaxation of \eqref{eq:sub}, $\mathcal{P}_{A_{TV}}$.

\begin{corollary}\label{cor:edges-are-connected}
    The edges of $\mathcal{P}_{A_{TV}}$ when understood as vectors are integral multiples of elements in $\mathfrak{G}_{A_{TV}}$. For vertices $(x_1, a_1)$ and $(x_2, a_2)$ to be adjacent, $x_1 - x_2$ must be an integral multiple of $\chi_S$ for some $S \subset \Vcal$, which induces a connected subgraph.
\end{corollary}

In fact, if our objective function $J$ in \eqref{eq:General-Problem} is linear and we run a simplex method on the linear relaxation, from Corollary \ref{cor:edges-are-connected}, this is effectively performing a Graver augmentation algorithm as the simplex method traverses the edges of $\mathcal{P}_{A_{TV}}$. We explore this observation further in Section \ref{sec:trust-region}.

\begin{table}[]
    \centering
    \begin{tabular}{c|c|c|c|c|c|c|c}
    n & 1 & 2 & 3 & 4 & 5 & 6 & 7 \\ \hline
	\# Connected Subgraphs & 1 & 13 & 218 & 11506 & 2301877 & 1.73 $\times 10^9$ & 4.87 $\times 10^{12}$  \end{tabular}
    \caption{Number of connected induced subgraphs of the $n \times n$ grid graph (from OEIS \cite[A059525]{oeis})}
    \label{tab:connected-graph-count}
\end{table}

\subsection{Variant for Total Variation Problem}
Since the total variation is one of the more common forms of $G$ in the objective of $\eqref{eq:General-Problem}$, we explore a specific linearization for this version of the problem.
Using standard techniques, we introduce new variables $a_{uv}^+, a_{uv}^-$ for each edge $uv$ where $|x_u - x_v| = a_{uv}^+ + a_{uv}^-$ and $x_u - x_v = a_{uv}^+ - a_{uv}^-$. This leads to the following integer linear program:

\begin{subequations}
\begin{align}
    \min_{x,a^+,a^-} \qquad &\sum_{v \in \Vcal} F_v(x_v) + \sum_{uv \in \Ecal} a_{uv}^+ + a_{uv}^- \\
    \text{s.t.} \qquad & x_u - x_v = a_{uv}^+ - a_{uv}^-, \quad uv \in \Ecal\\
                & x \in \{0,1,\ldots,Q\}^{\Vcal}, a^+, a^- \in \R_{\ge 0}^{\Ecal}.
\end{align} \label{eq:TV-Linear}
\end{subequations}
The constraint matrix of \eqref{eq:TV-Linear} is only a minor modification of the constraint matrix for \eqref{eq:sub} and is given by concatenating another identity matrix:
\begin{equation}
    A_{\TV}' = \begin{bmatrix}
        A_{TV} & I
    \end{bmatrix} = \begin{bmatrix} B & -I & I\end{bmatrix},\label{eq:tv-constraints}
\end{equation}
where $B$ is as before an oriented incidence matrix for the underlying graph $\Gcal$. Since $A_{TV}$ is totally unimodular and we are producing $A_{TV}'$ by concatenating an identity matrix to $A_{TV}$, it follows that $A_{TV}'$ is totally unimodular.  Therefore Remarks \ref{rmk:verts-are-integral} and \ref{rmk:edges-are-graver} hold for this problem, too.

Unlike in \eqref{eq:sub}, for each $x \in \{0,\ldots, Q\}^{\Vcal}$ there exist infinitely many choices of $a^+$ and $a^-$ that are feasible, since $(x, a^+ + Ce_{uv}, a^- + Ce_{uv})$ is feasible for $C > 0$ and any choice of $uv \in \Ecal$. However, only one choice of $a^+,a^-$ will satisfy $|x_u - x_v| = a_{uv}^+ - a_{uv}^-$, when $a_{uv}^+ = (x_u - x_v)_+$ and $a_{uv}^- = (x_u - x_v)_-$. 

\subsubsection{Graver Basis for Total Variation Variant}

The Graver basis for $A_{TV}'$ is by and large structurally the same as for $A_{TV}$ in that any move $(x, a^+, a^-) \in \Gfrak_{A_{TV}'}$ with $x \ne 0$ satisfies $x = \pm \chi_S$ for $S$ inducing a connected subgraph. The primary difference is that a multiplicity is introduced for each of the aforementioned moves. We describe how this multiplicity manifests itself in the Graver basis.

First we prove a lemma analogous to Lemma \ref{lem:TV-majorizer}.
\begin{lemma} \label{lem:TV-prime-majorizer}
    Suppose that $z = (x, a^+, a^-) \in \text{ker}_{\mathbb{Z}}(A_{TV}')$, where $x \ne 0$ where some $x_v > 0$. Let $c = \max_v x_v$, and let $S$ be a maximally connected component on the vertices where $x_v = c$. Then there exists $g \in \text{ker}_{\mathbb{Z}}(A_{TV}')$ where the projection of $g$ onto the $x$ variables is $\chi_S$ and $g \sqsubseteq z$. An analogous result holds when there exists $x_v < 0$, $c = \min_v x_v$, and $g$ projects to some $-\chi_S$.
\end{lemma}
\begin{proof}
    We write $g = (\chi_S, a'^+, a'^-)$. Since $\chi_S \subset x$, we only need to determine $a'^+$ and $a'^-$. Since $c$ is maximal and $S$ is maximally connected, we have that for $uv \in \delta^+(S)$, $x_u - x_v = a_{uv}^+ - a_{uv}^- \ge 1$. This implies that either $a_{uv}^+ \ge 1$ or $a_{uv}^- \le -1$. In the former case we may set $a'^+_{uv} = 1$ and $a'^-_{uv} = 0$, and in the latter case we set $a_{uv}'^+ = 0$ and $a'^-_{uv} = -1$. For edges $uv \in \delta^-(S)$, $x_u - x_v = a_{uv}^+ - a_{uv}^- \le -1$, so that $a_{uv}^+ \le -1$ or $a_{uv}^- \ge 1$. For the first case we take $a_{uv}'^+ = -1$ and $a_{uv}'^- = 0$, and in the second case we set $a_{uv}'^+ = 0$ and $a_{uv}'^- = 1$. For all other edges we take $a'^+_{uv} = a'^-_{uv} = 0$. By construction we have $a'^+ \sqsubseteq a^+$ and $a'^- \sqsubseteq a^-$, and all the constraints are satisfied by these choices,  showing that $g \in \text{ker}_{\mathbb{Z}}({A_{TV}'})$.
\end{proof}

From Lemma \ref{lem:TV-prime-majorizer} it follows that each Graver element either must have $x = 0$ or its nonzero $x$ values constitute a connected component. 
In the former case this imposes the restriction that $a_{uv}^+ = a_{uv}^-$ for all edges $uv$, implying that $\sqsubseteq$-minimal moves of this type are all of the form $\pm(0, e_{uv},e_{uv})$. 
In the latter case the moves take the form $g = \pm(\chi_S, a^+, a^-)$, where $S$ induces a connected subgraph. For each $uv \in \delta^+(S)$, we have $x_u - x_v = a_{uv}^+ - a_{uv}^- = 1$; and for each $uv \in \delta^-(S)$, $x_u - x_v = a_{uv}^+ - a_{uv}^- = -1$. We cannot have both $a_{uv}^+$ and $a_{uv}^-$ nonzero because such a move majorizes some $\pm(0, e_{uv}, e_{uv})$. Thus,  when $uv \in \delta^+(S)$, the only options are $a_{uv}^+ = 1$ and $a_{uv}^- = 0$ or $a_{uv}^+ = 0, a_{uv}^- = -1$; and when $uv \in \delta^-(S)$, we have $a_{uv}^+ = -1$ and $a_{uv}^- = 0$ or $a_{uv}^+ = 0$ and $a_{uv}^- = 1$. Each move of this type is denoted by a choice of which of $a_{uv}^+$ and $a_{uv}^-$ are 0 for each edge $uv$. We can thus characterize all moves of this form  first by a choice of connected subset $S$ and then by a partition of the $\delta(S) = E^+ \cup E^-$, where $E^+$ is the choice of edges where $a_{uv}^+$ is nonzero and $E^-$ is the choice where $a_{uv}^-$ is nonzero, as described above. Given $S, E^+, E^-$, we denote the corresponding Graver move by $z_{S,E^+,E^-}$. 
For each connected $S \subset V$, there are $2^{|S|}$ such moves. 
We sum up this discussion in the following theorem.

\begin{theorem}
    The Graver basis for $A'_{TV}$ comprises two sets of elements, $\Gfrak_{A_{TV}'} = \Gfrak_1 \cup \Gfrak_2$, where $\Gfrak_1 = \{ \pm (0, e_{uv}, e_{uv}) : uv \in \Ecal\}$, and $\Gfrak_2 = \{ \pm z_{S,A,B} : S \subset \Vcal, A \cup B = \delta(S), A \cap B = \emptyset, S \text{ induces a connected subgraph} \}$.
\end{theorem}\label{thm:alt-graver-basis}

\subsection{Augmentation Subproblems}

We now comment on how we may set up the augmentation subproblems in Algorithm \ref{alg:NL-Augmentation}.
In this setting we have a candidate point $\hat{x} \in \{0,1,\ldots,Q\}^{\Vcal}$ and we wish to search  the set of moves $\chi_X$ with $X \subset \Vcal$ for a move to either increment or decrement $\hat{x}$. 
We exploit the total unimodularity of the constraints to pose two linear programs for these subproblems.

For the incrementing problem we introduce a new binary variable $\Delta x_v \in \{0,1\}$ indicating whether we increment the $\hat{x}_v$ at a given vertex $v \in \Vcal$. We also include corresponding variables $\Delta a^+_{uv}, \Delta a^-_{uv} \in \{0, 1\}$, which have an analogous interpretation for incrementing variables $a^+_{uv}$ and $a^-_{uv}$. To maintain feasibility after augmentation, we require $\Delta x_u - \Delta x_v = \Delta a_{uv}^+ - \Delta a_{uv}^-$ for each $uv \in \Ecal$.
We represent the change in objective if $\hat{x}_v$ is incremented by $(F_v(\hat{x}_v + 1) - F_v(\hat{x}_v))\Delta x_v$. We abbreviate this coefficient as $\Delta^+F(\hat{x}_v) := F_v(\hat{x}_v + 1) - F_v(\hat{x})$ and denote $\Delta^-F(\hat{x}_v) := F_v(\hat{x}_v - 1) - F_v(\hat{x}_v)$ analogously.
Unlike $F$, to adapt the change in $G$ into our objective, we need to account for the fact that the argument $x_u - x_v$ may increase or decrease in value by 1.
The fact that we have two variables $a_{uv}^+$ and $a_{uv}^-$ for each edge provides the flexibility to incorporate this information. 
With this intuition, we let the coefficients of $\Delta a^+_{uv}$ and $\Delta a^-_{uv}$ be $\Delta^+G_{uv}(\hat{a}_{uv})$ and $\Delta^- G_{uv}(\hat{a}_{uv})$, respectively where $\hat{a}_{uv} := \hat{x}_u - \hat{x}_v$.

We then have the following formulation of the \textbf{up augmentation subproblem} \eqref{eq:Alt-Up-Augment}:
\begin{subequations}
\begin{align}
    \min_{\Delta x,\Delta a^+,\Delta a^-} \qquad & \sum_{v \in \Vcal} \Delta^+ F_v(\hat{x}_v)\Delta x_v + \sum_{uv \in \Ecal}\Delta^+ G_{uv}(\hat{a}_{uv})\Delta a_{uv}^+ + \Delta^-G_{uv}(\hat{a}_{uv})\Delta a_{uv}^-  \\
    \text{s.t.} \qquad& \Delta x_u - \Delta x_v = \Delta a_{uv}^+ - \Delta a_{uv}^-, \quad uv \in \Ecal \\
    &\Delta x \in \{0,1\}^{\Vcal}, \Delta a^+, \Delta a^- \in \{0,1\}^{\Ecal}.
\end{align} \label{eq:Alt-Up-Augment}%
\end{subequations}
The \textbf{down augmentation subproblem} is presented in \eqref{eq:Alt-Down-Augment}. In this context, if $\Delta x_v=1$, we decrement $x_v$, and therefore its coefficient in the objective is $\Delta^-F_v(\hat{x}_v)$. $\Delta^+a_{uv}$ and $\Delta a^-_{uv}$ maintain the same interpretation, but note that we need to reverse $\Delta a^-_{uv}$ and $\Delta a^+_{uv}$ in \eqref{eq:down-aug-eqn}.
\begin{subequations}
\begin{align}
    \min_{\Delta x,\Delta a^+,\Delta a^-} \qquad & \sum_{v \in \Vcal} \Delta^-F_v(\hat{x})\Delta x_v + \sum_{uv \in \Ecal}\Delta^+G_{uv}(\hat{a}_{uv})\Delta a_{uv}^+ + \Delta^-G_{uv}(\hat{a}_{uv})\Delta a_{uv}^- \\
    \text{s.t.} \qquad& \Delta x_u - \Delta x_v = \Delta a_{uv}^- - \Delta a_{uv}^+, \quad uv \in \Ecal \label{eq:down-aug-eqn}\\
    &\Delta x \in \{0,1\}^{\Vcal}, \Delta a^+, \Delta a^- \in \{0,1\}^{\Ecal}
\end{align} \label{eq:Alt-Down-Augment}%
\end{subequations}
Implicit in this formulation is an assumption of complementarity for variables $\Delta a^+_{uv}$ and $\Delta a^-_{uv}$. 
As is, the formulation allows both $\Delta a_{uv}^+$ and $\Delta a_{uv}^-$ to be nonzero, which is contradictory to our assumptions in the problem since if we increment both $a^+_{uv}$ and $a^-_{uv}$, we expect no change in $x_u - x_v$, but in the formulation we incur the cost $\Delta^+ G_{uv}(\hat{a}_{uv}) + \Delta^- G_{uv}(\hat{a}_{uv})$ in the objective. 
However, we do not need to be concerned about this case because we assume each $G_{uv}$ convex, which then implies
\begin{equation}
    \Delta^+ G_{uv}(\hat{a}_{uv}) + \Delta^- G_{uv}(\hat{a}_{uv}) \ge 0,
\end{equation}
and therefore any solution where both variables are nonzero can be replaced with a solution with complementarity enforced of equal or lower objective.

Each of these subproblems may be solved as linear programs given the total unimodularity of the constraints. This is not the fastest way to solve these problems given that they can be set up as minimum cut problems (as in \cite{kolmogorov2009new}).
However, we explicitly pose them in this way because applying primal simplex methods to these problems now corresponds to Graver augmentation methods owing to Corollary \ref{cor:edges-are-connected}.
We will see in Section \ref{sec:trust-region} how we can use these subproblems to produce randomized methods for sampling Graver augmentations even for the fully constrained problem.

\section{Randomized Heuristics for Constrained Problem}\label{sec:trust-region}

Now we consider problem \eqref{eq:General-Problem} in full. 
Our approach to the constrained problem largely follows that of the unconstrained problem in that we set up the two associated augmentation subproblems \eqref{eq:Alt-Up-Augment} and \eqref{eq:Alt-Down-Augment}, and we repeatedly search the Graver basis $\Gfrak_{A_{TV}}$ for improving moves until we cannot produce any more.
Unfortunately, with the additional constraint, the Graver basis no longer functions as an optimality certificate, and this approach will not guarantee a globally optimal solution.
Given the enormity of the Graver basis, however, we hope this will provide a reasonable heuristic for producing high-quality solutions.

In this section we first assess notions of local optimality when using a Graver basis. Then we demonstrate a method for sampling improving feasible moves from this set via a variant of the simplex method. We compose this sampling method with initialization techniques to produce a randomized augmentation algorithm for \eqref{eq:General-Problem}.

\subsection{Local Graver Optimality}

We say that a feasible point $x$ for \eqref{eq:General-Problem} is $\Gfrak_{A_{TV}}$\textbf{-optimal} if for all $g = (g_x, g_a) \in \Gfrak_{A_{TV}}$, we have that either $J(x + g_x) \ge J(x)$ or $x + g_x$ is infeasible. Note that in the absence of the budget constraint, a $\Gfrak_{A_{TV}}$-optimal point is also globally optimal.

We compare this notion of local optimality with other common forms. We say a point $x$ is \textbf{$k$-optimal} if for all other feasible $x'$ with $\|x - x'\|_1 \le k$, we have $J(x) \le J(x')$. Since the Graver basis projected on its $x$ variables includes all moves of the form $\pm e_v$, we have the following proposition.
\begin{prop}
    If $x$ is $\Gfrak_{A_{TV}}$-optimal for \eqref{eq:General-Problem}, then it is also 1-optimal.
\end{prop}
However, a point can be $\Gfrak_{A_{TV}}$-optimal and not 2-optimal since the ball of radius 2 about $x$ includes points that involve both incrementing and decrementing to reach and all Graver moves solely increment or decrement. $\Gfrak_{A_{TV}}$-optimality may appear fairly weak in that it does not guarantee optimality over a relatively small neighborhood of a feasible point.
Nevertheless, using $\Gfrak_{A_{TV}}$ offers advantages in navigating the integer lattice.
In particular, we can guarantee there is a sequence of improving Graver moves to the global optimum given that we start at the correct point as described by Proposition \ref{prop:minimal-budget}. In addition, we demonstrate empirically in Section \ref{sec:experiments} that we can often achieve an objective close to optimal when using the Graver basis.

\begin{prop}\label{prop:minimal-budget}
    Let $x^{(0)}$ be given such that $x^{(0)}_v$ minimize $H_v$ over $\{0,1,\ldots,Q\}$
    There exists a sequence of Graver elements $g^{(0)},\ldots,g^{(N-1)}$ and iterates $x^{(k+1)} = x^{(k)} + \alpha_kg^{(k)}$, $\alpha_k \in \Z_{> 0}$ for $k=0,\ldots,N-1$ such that $\{J(x^{(k)})\}_k$ is a decreasing sequence, $H(x^{(k)}) \le \Delta$ for all $k$ and $x^{(N)}$ is the globally optimal solution of \eqref{eq:General-Problem}.
\end{prop} 
\begin{proof}
    Let $x^*$ be an optimal solution of \eqref{eq:General-Problem}. 
    We decompose $x^* - x^{(0)} = \sum_{k=1}^{N}\alpha_k g^{(k)}$ into Graver moves where $\alpha_k \ge 0$ and $g^{(k)} \sqsubseteq x^* - x^{(0)}$. 
    From \cite{de2012algebraic} we can order the Graver basis elements such that the sequence of iterates $x^{(i)} = x^{(0)} + \sum_{k=1}^i\alpha_k g^{(k)}$ satisfies $J(x^{(i)}) \le J(x^{(i-1)})$ for $i=1,\ldots,N$. Hence, we only need to verify that all $x^{(i)}$ are feasible. Since each $H_v$ is convex and $x^{(0)}_v$ is a minimizer for $H_v$, we have that $H_v$ is decreasing on $[0, x^{(0)}_v]$ and increasing on $[x^{(0)}_v, Q]$. As each $g^{(k)} \sqsubseteq x^* - x^{(0)}$, it necessarily follows that for $v \in \Vcal$, $g^{(1)}_v, g^{(2)}_v, \ldots, g^{(N)}_v$ are all nonnegative or all nonpositive. Hence, the sequence $\{|x^{(i)}_v - x^{(0)}_v|\}_{i}$ is increasing for each $v \in \Vcal$, and therefore $\{H_v(x^{(i)}_v)\}_i$ is also increasing. We have for $i=0,1,\ldots,N$,
    \begin{equation}
        H(x^{(i)}) = \sum_{v \in \Vcal}H_v(x^{(i)}_v) \le \sum_{v\in \Vcal}H_v(x^{(N)}_v) = H(x^{(N)}) = H(x^*) \le \Delta ,
    \end{equation}
    where the final inequality follows from the fact that $x^*$ is feasible for \eqref{eq:General-Problem}. Therefore, we have produced a sequence of feasible iterates decreasing in objective and achieving the global minimum.
\end{proof}

% Now we examine in further detail a specific variant of the total variation-regularized optimization problem where we add a trust region constraint to the problem. This problem arises as a subproblem of a trust region method for solving the general nonlinear problem \eqref{eq:TV-Nonlinear}. In this setup, given a feasible point $x_k$, we approximate $F$ by a model $m_k$ (often quadratic or linear) which is easier to optimize over. We then solve the following subproblem:

% \begin{equation}
% \begin{aligned}
%     \min_{x} \qquad& \nabla F(\hat{x})^Tx + \TV(x) \\
%     \text{s.t.} \qquad& \|x - \hat{x} \|_1 \le \Delta \\
%     & x \in \{0,1,\ldots,Q\}^{|\Vcal|}
% \end{aligned} \tag{TV-Trust-Region} \label{eq:TV-Trust-Region}
% \end{equation}
% where $\Delta_k \in \Z$ is the trust region radius. 

\subsection{Sampling Improving Graver Moves}
From Corollary \ref{cor:edges-are-connected}, all edges of the underlying polytope $\mathcal{P}_{{A}_{TV}}$ in problems \eqref{eq:Alt-Up-Augment} and \eqref{eq:Alt-Down-Augment} correspond to integer multiples of Graver moves. Since applying the simplex method to a linear program corresponds to traversing the edges of this polytope, the simplex method in this context is a specific instantiation of a Graver augmentation algorithm. 
We can take this insight a step further to motivate a sampling algorithm for improving Graver moves. The exponential size of the Graver basis relative to the number of constraints in the description of $\mathcal{P}_{A_{TV}}$ suggests that $\mathcal{P}_{A_{TV}}$ is extremely degenerate. This suggests that there are several basic solutions for each degenerate vertex and that pivots often involve an arbitrary choice from several valid exiting basic variables. We can randomize the algorithm with a random selection of an initial basic solution and randomize the choice of leaving variables. In this fashion, we will produce random improving Graver moves.

From our prior work in \cite{yang2025specialized}, we can associate basic solutions with rooted spanning forests. The association goes as follows: We construct a graph $G_{NB} = (V, E_{NB})$, where $E_{NB}$ comprises all edges $uv$ where both $a_{uv}^+$ and $a_{uv}^-$ are nonbasic. In Proposition \ref{prop:tv-linear-basic} it is shown that this graph is a forest, and furthermore there is only one nonbasic vertex per tree, meaning they can be interpreted as roots. Moreover, each non-tree edge $uv$ has an orientation dictating which of $a_{uv}^+$ and $a_{uv}^-$ is basic; and if these orientations are aligned with the value of $x$ (i.e., if $x_u > x_v$, then the arrow points from $u$ to $v$), the solution is feasible. Some visualizations of the rooted spanning forest representation of basic solutions are presented in Figure \ref{fig:rooted-spanning-forests}. Proposition \ref{prop:tv-linear-basic} is a simpler version of a proposition in \cite{yang2025specialized}. We include a slightly different proof of this fact for completeness.

\begin{prop} \label{prop:tv-linear-basic}
    Let $x$ be a basic solution of \eqref{eq:TV-Linear}. Let $\Vcal_{NB} \subset \Vcal$ denote the vertices for which $x_v$ are nonbasic, and let $\Ecal^+, \Ecal^- \subset \Ecal$ denote the edges $uv$ for which $a_{uv}^+$ and $a_{uv}^-$ are nonbasic, respectively. Let $\Gcal_{NB} = (\Vcal, \Ecal^+ \cap \Ecal^-)$ be the subgraph given by taking only edges for which both $a_{uv}^+$ and $a_{uv}^-$ are nonbasic, and let $\Vcal = \Vcal_1 \cup\dots \cup \Vcal_k$ be a partition of the vertices into connected components. Then the following statements are true:
    \begin{enumerate}
        \item $|\Vcal_{NB}| + |\Ecal^+| + |\Ecal^-| = |\Vcal| + |\Ecal|$.
        \item $\Ecal^+ \cup \Ecal^- = \Ecal$.
        \item The subgraph $\Gcal_{NB} = (\Vcal, \Ecal^+ \cap \Ecal^-)$ is a spanning forest. 
        \item Each tree in $\Gcal_{NB}$ has exactly one vertex $v$ for which $x_v$ is nonbasic.
    \end{enumerate}
\end{prop}
\begin{proof}
    Statement 1 is immediate from the fact that there are $|\Vcal| + 2|\Ecal|$ variables and $|\Ecal|$ equality constraints; hence any basic solution has $|\Ecal|$ basic variables and $|\Vcal| + |\Ecal|$ nonbasic variables. Statement 2 follows from the observation that the variables $a_{uv}^+$ and $a_{uv}^-$  appear only in one constraint $x_u - x_v = a_{uv}^+ - a_{uv}^-$. For any choice of these variables, $\bar{a}^{+}_{uv} = a_{uv}^+ + C, \bar{a}^-_{uv} = a_{uv}^- + C$ also satisfy the constraints indicating the system is singular. Hence, at least one of these variables must be nonbasic and equal 0.

For Statements 3 and 4 we observe that if both $a_{uv}^+$ and $a_{uv}^-$ are nonbasic, then $x_u = x_v$. Therefore, on connected components  of $\Gcal_{NB}$, $x$ is constant. For each edge $uv$ between components, the basic variable from $a_{uv}^+$ and $a_{uv}^-$ will take the value of the (appropriately signed) difference of the $x$ values of each component. All that remains to determine the value of all the variables is to establish the value of each component. If in each component there is one vertex $v$ for which $x_v$ is nonbasic, the value of $x_v$ determines the value of the component; but if there is no such variable, each choice of value gives a solution satisfying the constraint equations. Hence, for each component there must be one vertex $v$ for which $x_v$ is nonbasic. There cannot be more than one such vertex because if they took on distinct values, the linear system would be inconsistent and hence such a choice of nonbasic variables would not produce a nonsingular matrix.

To see why each component must be a tree, we proceed with a counting argument. Since for each edge $uv$ one of $a_{uv}^+$ and $a_{uv}^-$ is nonbasic, this covers $|\Ecal|$ variables, and the remaining $|\Vcal|$ nonbasic variables come from picking vertices $v$ where $x_v$ is nonbasic and edges where both $a_{uv}^+$ and $a_{uv}^-$ are nonbasic. If each component of $\Gcal_{NB}$ is a tree, each component has $|\Vcal_i| - 1$ edges, and altogether there are $\sum_{i=1}^k (|\Vcal_i| - 1) = |\Vcal| - k$, which leaves $k$ nonbasic variables that can be selected one from each component. If one component contains a cycle, there is necessarily more than $|\Vcal| - k$ edges chosen to be nonbasic, which does not leave enough nonbasic $x_v$ to cover all $k$ components.
\end{proof}

\begin{figure}
    \centering
    \includegraphics[width=\linewidth]{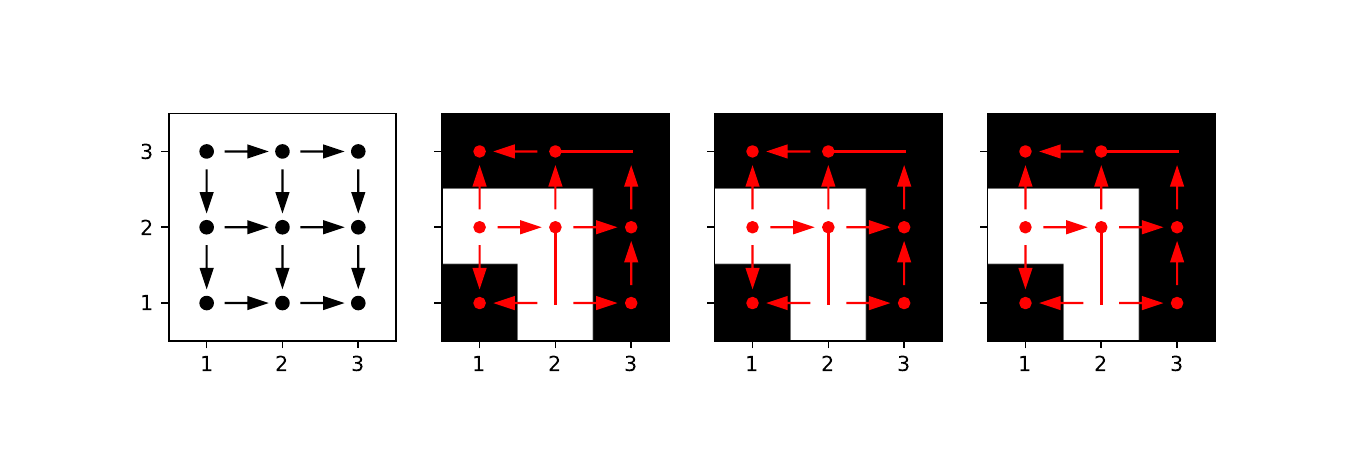}
    \caption{Rooted spanning forest representation of three consecutive basic solutions in the simplex algorithm on a $3\times 3$ grid graph (shown on the left). Vertices $v$ with a red circle indicate nonbasic $x_v$; all other $x_v$ are basic. Red undirected edges $uv$ indicate that both $a_{uv}^+$ and $a_{uv}^-$ are nonbasic. If the direction of an arrow is aligned with the underlying graph, $a_{uv}^+$ is basic and $a_{uv}^-$ is nonbasic, and the reverse is true for the opposite orientation. White cells indicate the variable $x_v = Q$, and black cells indicate $x_v = 0$.}
    \label{fig:rooted-spanning-forests}
\end{figure}

Proposition \ref{prop:tv-linear-basic} now gives an easy way of randomizing the initialization of the simplex method. Assuming we start with the solution $\Delta x_v = \Delta a_{uv}^+ = \Delta a_{uv}^- =  0$ for \eqref{eq:Alt-Up-Augment} or \eqref{eq:Alt-Down-Augment}, sampling any rooted spanning forest of $\Gcal$ alongside a random choice as to which of $\Delta a_{uv}^+$ and $\Delta a_{uv}^-$ is basic gives a valid basic feasible solution. In practice, we choose the spanning forest where each vertex is its own tree, and we uniformly at random pick which of $\Delta a_{uv}^+$ and $\Delta a_{uv}^-$ is basic for each edge $uv$.

We further randomize the algorithm for any pivot by selecting from the valid exiting basic variables uniformly at random. Given the extreme degeneracy noted in \cite{yang2025specialized}, there are typically several. From the rooted spanning forest representation of basic solutions and following \cite{yang2025specialized}, we may interpret these pivots as merging, splitting, and shifting the value on trees in the forest in a randomized fashion. Edge variables entering and leaving the basis can be understood as splitting and merging trees along that given edge, and vertex variables leaving and entering the basis indicate marking and unmarking that vertex as the root of its tree.

The rooted spanning forest representation also gives an easy indication of how much a given pivot is predicted to change the budget constraint. 
Bringing a nonbasic $x_v$ into the basis corresponds to attempting to shift the value of the entire tree $T_v$, where $v$ is the root up or down by 1. 
For $uv$ directed away from the root of the tree, introducing $a_{uv}^+$ to the basis is done by attempting to shift the subtree of $v$ and its descendants down by 1, therefore increasing $a_{uv}^+ = x_u - x_v$. 
Likewise, introducing $a_{uv}^-$ can be done by shifting the subtree $T_v$ up by 1. 
These predicted changes are recorded in \eqref{eq:delta-H}, where we define $\Delta H_z$ as the predicted change in budget for introducing a given nonbasic variable to the basis. For edges $uv$ directed toward the root, the reverse is true.

\begin{equation} \label{eq:delta-H}
    \Delta H_z = \begin{cases}
        \sum_{w \in T_v} H(x_w + 1) - H(x_w) & z = a_{uv}^-, \text{ or } z = x_v, x_v = 0\\
        \sum_{w \in T_v} H(x_w - 1) - H(x_w) & z = a_{uv}^+, \text{ or } z = x_v, x_v = 1
    \end{cases}
\end{equation}
By computing $\Delta H_z$ for each variable, we can reject any pivot that is predicted to leave the feasible region.
Putting these parts together, we can produce a sampling algorithm that produces only improving moves within the feasible region, as is done in Algorithm \ref{alg:Randomized-TV-Algo}.

\SetKwFunction{SampleGraverBasis}{SampleGraverBasis}
\begin{algorithm}
\caption{Graver Basis Sampling Algorithm for \eqref{eq:Alt-Up-Augment} and \eqref{eq:Alt-Down-Augment}}\label{alg:Randomized-TV-Algo}
\DontPrintSemicolon
\Fn{\SampleGraverBasis{$x$, $F$, $G$, $H$, $p$, direction}}{
%\KwData{$x \in \{0,1,\ldots,Q\}^{\Vcal}$, $F$, $G$, $H$, $p \ge 0$}
\leIf {direction = up}{Randomly initialize \eqref{eq:Alt-Up-Augment}}{Randomly initialize \eqref{eq:Alt-Down-Augment}}
%Randomly initialize Simplex method on \eqref{eq:Alt-Up-Augment} or \eqref{eq:Alt-Down-Augment} starting from $x$ using $F$ and $G$\;
\While{$\exists$nonbasic variable $z$ with reduced cost $r_z < 0$ and $H(x) + \Delta H_z \le \Delta$}
{
    Select entering variable $z$ satisfying $H(x) + \Delta H_z \le \Delta$ minimizing $r_z / w(\Delta H_z,p)$\;
    Choose exiting variable from available candidates uniformly at random\;
    \If(\Comment*[f]{This is a sampled move}){Pivot is nondegenerate} 
    { Let $g$ be corresponding Graver move\;
        \Return $g$\;
    }
    Perform pivot\;
}
}
\end{algorithm}

We may also bias the sampling by selecting variables other than those with the most negative reduced cost. One natural adjustment is to weight pivots by how much they impact budget usage. A method that we  use  in Section \ref{sec:experiments} is to weight the reduced cost for variable $z$ by $\Delta H_z$, as in the following expression:
\begin{equation} \label{eq:weighting-scheme}
    w(\Delta H_z,p) = \begin{cases}
        1 & \Delta H_z \le 0 \\
        (1+\Delta H_z)^p & \Delta H_z > 0.
    \end{cases}
\end{equation}
When $p = 0$, this corresponds to the standard method of choosing the variable with most negative reduced cost. As $p$ increases, however, there is a greater bias toward pivots that use the budget efficiently.

The algorithm stops once all possible pivots are predicted to either increase objective value or leave the feasible region. In the absence of the constraint $H(x) \le \Delta$, this is equivalent to optimality; but as is, this gives an interesting version of local optimality that is not Graver local optimality but asserts that there exists a rooted spanning forest where shifting the values on any subtree of the forest will either increase the objective or leave the feasibility region.

Like other simplex methods, we do not have any proof that this algorithm will terminate in polynomial time. In Section \ref{sec:experiments}, however, we will show empirically that this method often produced high-quality solutions and is usually significantly faster than standard integer program solvers. In particular, an interesting relationship exists between the choice of $p$ and the number of pivots; we explore this in Section \ref{sec:p-relation}.

\subsection{Randomized Heuristic for the General Problem}

The ability to sample the Graver basis now provides the foundation for an algorithm for producing approximate solutions to \eqref{eq:General-Problem}. 
To this end, up and down augmentation problems \eqref{eq:Alt-Up-Augment} and \eqref{eq:Alt-Down-Augment} again form a basis for our optimization routine.
Sampling from the Graver bases according to these subproblems leads to a heuristic for solving the problem via randomized augmentation.
In Algorithm \ref{alg:Randomized-Augmentation} we start with an initial feasible point $x^{(0)}$, the choice of which we discuss later in this section. 
Then we repeatedly sample from the set of improving incrementing and decrementing moves using Algorithm \ref{alg:Randomized-TV-Algo} until the simplex method reports no improving moves for each problem. 
It is possible for the sampling algorithm to fail to find an improving move, in which case it returns $\Delta x = 0$.
We note that the alternating scheme for sampling moves is arbitrarily chosen. It may be better to perform several incrementing moves before a decrementing move for certain problems, but we defer this study for future work.

We finish this algorithm with a polishing procedure to ensure that the given solutions are 2-optimal. 
This involves iterating over all pairs of vertices to see whether the solution may be improved by adding moves of the form $a_i e_i + a_j e_j$ for $a_i, a_j \in \{-1, 1\}$. Once no more moves of this form can be applied, the algorithm terminates.

\SetKwFunction{PolishTwoOpt}{Polish2Opt}
\SetKwFunction{RandomizedAugmentation}{RandomizedAugmentation}
\begin{algorithm}
\caption{Randomized Augmentation Algorithm for \eqref{eq:General-Problem}}\label{alg:Randomized-Augmentation}
\Fn{\RandomizedAugmentation{$x$, $F$, $G$, $H$, $p$}}{
\KwData{$x^{(0)} \in \{0,1,\ldots,Q\}^{\Vcal}$, $F, G, H$ $p \ge 0$}
$k \gets 0$\;
\While{True}
{
    $\Delta x^{\text{up}} \gets $ \SampleGraverBasis{$x^{(k)}$, $F$, $G$, $H$, $p$, up}\;
    $x^{(k+1)} \gets x^{(k)} + \Delta x^{\text{up}}$\;
    $\Delta x^{\text{down}} \gets $ \SampleGraverBasis{$x^{(k+1)}$, $F$, $G$, $H$, $p$, down}\;
    $x^{(k+2)} \gets x^{(k+1)} + \Delta x^{\text{down}}$\;
    \If(\Comment*[f]{Neither Augmentation changed $x^{(k)}$}){ $x^{(k+2)} = x^{(k)}$}
    {
        
        Exit loop\; 
    }
    $k \gets k + 2$\;
}
$\tilde{x}^{(k)} \gets \PolishTwoOpt(x^{(k)})$\;
\Return $\tilde{x}^{(k)}$\;
}
\end{algorithm}
\subsubsection{Picking an Initial Point}\label{sec:starting-point}
We elaborate on choices of initial point $x^{(0)}$ in Algorithm \ref{alg:Randomized-Augmentation}. Since the success of the algorithm relies on repeated sampling of ``good'' moves, an initial point that is closer to the optimal solution and thus requires fewer random guesses can significantly improve performance.

\paragraph{Minimal Budget Point.}
A natural starting point $x^{(0)}$ is one that minimizes use of budget, that is, $x^{(0)} = \min_{x \in \{0,1,\ldots,Q\}^{\Vcal}} H(x)$. Since $H$ is separable and convex, this point can be found efficiently by simple bisection algorithms in each dimension. In view of Proposition \ref{prop:minimal-budget}, this choice also has the advantage of guaranteeing the algorithm has a positive probability (however small) of finding the optimal solution.

\paragraph{Penalty-Based Approach.}
Our next choice of initialization relaxes the budget constraint and introduces it as a penalty term with weight $\mu > 0$ as follows:
\begin{subequations}
\begin{align}
    \min_{x} \qquad& \sum_{v\in \Vcal}F_v(x_v) + \sum_{uv\in \Ecal}G_{uv}(x_u - x_v) + \mu \left(\sum_{v \in \Vcal} H_v(x_v) - \Delta\right) \\
    \text{s.t.} \qquad& x \in \{0,1,\ldots,Q\}^{\Vcal} .
\end{align}\label{eq:penalty-formulation}%
\end{subequations}
This may be solved efficiently by using Algorithm \ref{alg:graver-augmentation} since it is of the form \eqref{eq:General-Problem}. 
When $G$ is the total variation, $F$ linear, and $H$ of the form $H(x) = \|x-x^*\|_1$, the authors in \cite{manns2024discrete} give a conditional $p$-approximation result that asserts that if the solution $\bar{x}$ to \eqref{eq:penalty-formulation} is feasible and uses up a proportion $p$ of the budget $\Delta$, it follows that $J(\bar{x}) - J(\hat{x}) \le p(J(x^*) - J(\hat{x}))$. Since their proof makes no use of the specific form of $F$, $G$, or $H$,  we can generalize their result.
\begin{prop}\label{prop:penalty-form}
    Let $x^{(0)} \in \{0,1,\ldots,Q\}^{\Vcal}$ minimize $H$ where $H(x^{(0)}) = 0$. Let $\bar{x}$ be the optimum of \eqref{eq:penalty-formulation}, and let $x^*$ be the optimum of \eqref{eq:General-Problem}. If we have that $p\Delta \le H(\bar{x})\le \Delta$, it follows that $J(\bar{x}) - J(x^{(0)}) \le p(J(x^*) - J(x^{(0)}))$.
\end{prop}
\begin{proof}
    Since $\bar{x}$ is optimal for \eqref{eq:penalty-formulation}, we have $J(\bar{x}) + \mu H(\bar{x}) \le J(x^*) + \mu H(x^*)$. We therefore have
    \begin{align}
    J(\bar{x}) &\le J(x^*) + \mu (H(x^*) - H(\bar{x})) \\
        &\le J(x^*) + \mu(1 - p)\Delta , \label{eq:J2}
    \end{align}
where in the second inequality we make use of how $H(x^*)\le \Delta$ and that $H(\bar{x}) \ge p\Delta$. 
If $p = 1$, we have $J(\bar{x}) \le J(x^*)$ as desired. Suppose alternatively that $p < 1$. 
If we have $\mu(1-p)\Delta \le (p-1)(J(x^*) - J(x^{(0)}))$, or equivalently $-\mu\Delta \ge J(x^*) - J(x^{(0)})$, then \eqref{eq:J2} implies $J(\bar{x}) - J(x^{(0)}) \le p(J(x^*) - J(x^{(0)}))$ as desired. 
Assume instead that $-\mu\Delta < J(x^*) - J(x^{(0)})$. 
Since $x^{(0)}$ is feasible for \eqref{eq:penalty-formulation} and $\bar{x}$ is optimal, we have $J(\bar{x}) + \mu(H(\bar{x}) - \Delta)\le J(x^{(0)})-\mu(H(x^{(0)}) - \Delta) = J(x^{(0)}) - \mu \Delta$. Since $H(\bar{x}) - \Delta \ge (p-1)\Delta$, we have
    \[
        J(\bar{x})  - J(x^{(0)}) \le -\mu p\Delta < p(J(x^*)-J(x^{(0)})) ,
    \]
which gives us our desired approximation result.
\end{proof}

Determining $\mu$ involves finding the smallest $\mu$ such that the solution of \eqref{eq:penalty-formulation} is feasible for \eqref{eq:General-Problem}. This may be done by using a bisection algorithm to the desired precision in $\mu$.

\paragraph{Rounded Solution to Continuous Relaxation.}
Another natural approach to initialization is to solve the continuous relaxation of \eqref{eq:General-Problem} and appropriately round the solution so as to remain feasible. 
This is always possible if we assume $H_v$ is convex and minimized at an integer.  If $x_v$ is nonintegral at a solution $x$, either $H_v(\lfloor x_v \rfloor) \le H_v(x_v)$ or $H_v(\lceil x_v \rceil) \le H_v(x_v)$, meaning there exists a rounding that is feasible. 
This approach will not necessarily give good solutions if the solution to the continuous relaxation has few integral values. However, when $F$ is linear,$G$ is the total variation, and $H(x) = \|x - \hat{x}\|_1$, the authors in \cite{manns2024discrete} show that the optimal solution of the linear relaxation is integral everywhere except on an induced connected subgraph upon which it takes a constant nonintegral value. Often this subgraph is a relatively small component of the graph, in which case the rounded solution gives a good approximation of the optimal solution.

For both of these latter approaches, we do not have a guarantee like Proposition \ref{prop:minimal-budget} as to the existence of a sequence of improving Graver moves from these initial points to the optimal point. In fact, we may be precluding the possibility of finding the optimal solution should we use these alternate initializations. Nevertheless, we observe in Section \ref{sec:experiments} that these choices  often lead to higher-quality solutions.

\section{Computational Experiments}
\label{sec:experiments}
In this section we explore empirically the efficiency of our algorithms on the two total variation-regularized problems presented in the Introduction: the image reconstruction problem \eqref{eq:image-reconstruction} and the trust-region--constrained subproblem \eqref{eq:trust-region-subproblem}. On each of these problem sets we demonstrate the performance of our own implementation of Algorithm \ref{alg:Randomized-Augmentation} under a selection of parameterizations, and we study the impact of different initializations and varying the efficiency parameter $p$. All examples are benchmarked against the state-of-the-art integer program solver CPLEX.

\subsection{Computational Setup}

In both of our test cases we consider problem instances on $N\times N$ grids with $N \in \{32, 64, 96\}$.  For the image reconstruction problems, we use the images given in Figure \ref{fig:images}, which are a selection of geometric and real images. Each image is resized to be an $N\times N$ image, and the intensity at each pixel is rescaled to take value in $[0, Q]$ where $Q=3$. For these problems we consider $\alpha_i = 0.25i$ for $i=1,2,3,4$ (i.e., $\alpha \in \{0.25,0.5,1,2\})$, and we vary $\Delta$ based on the solution to the unconstrained problem. For each image and choice of $N$ and $\alpha$, we solve \eqref{eq:image-reconstruction} using Algorithm \ref{alg:graver-augmentation} and record the budget $\Delta^*$ used by the optimal solution. Then, to ensure that the budgetary constraint actually impacts the optimal solution, we consider instances where $\Delta = \delta \Delta^*$ for $\delta \in \{0.25, 0.5, 0.75, 1\}$. Altogether, we have 288 instances with 24 instances for each choice of $N$ and $\alpha$.

\begin{figure}[t]
	\centering
	\includegraphics[width=0.25\linewidth]{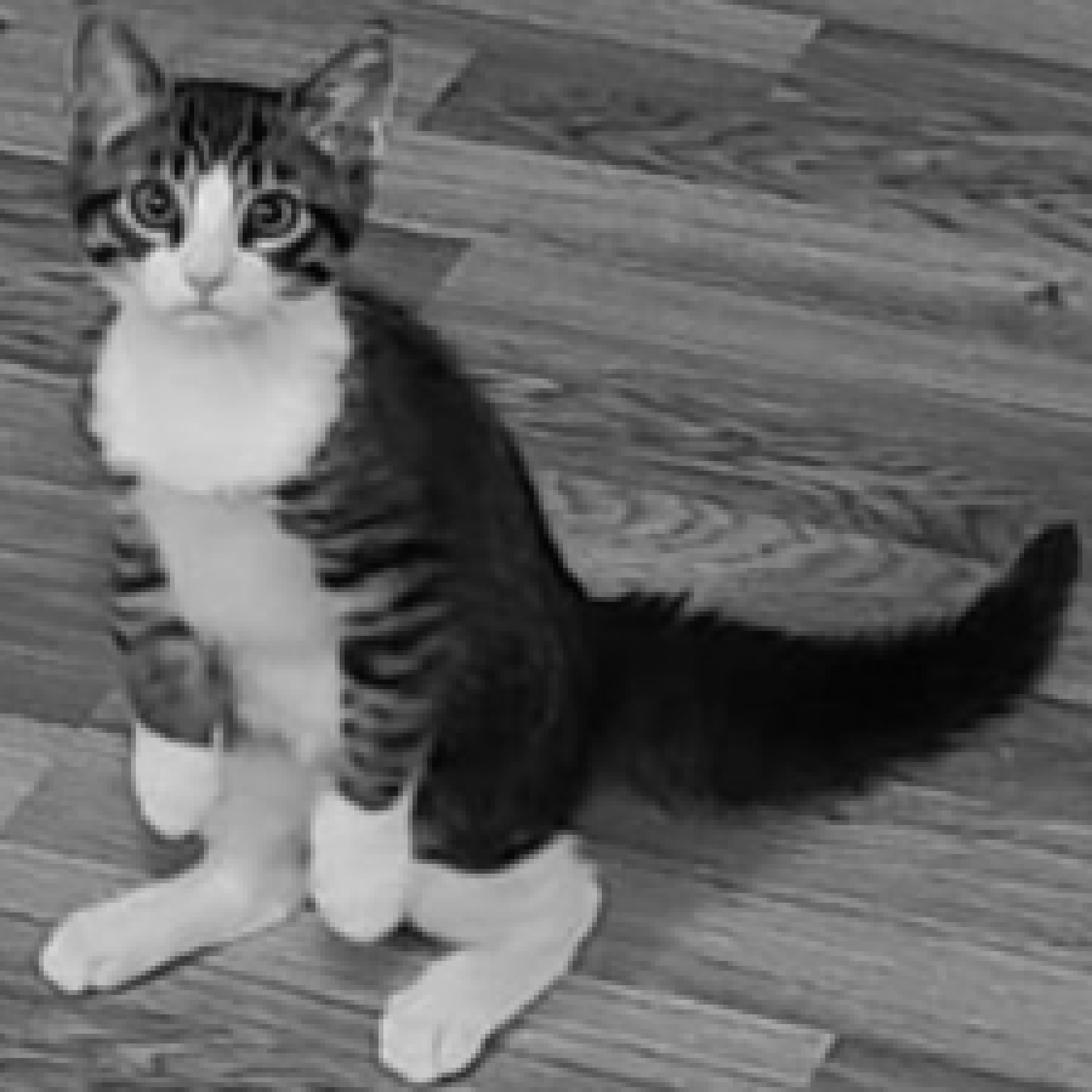}\hspace{0.1cm}
	\includegraphics[width=0.25\linewidth]{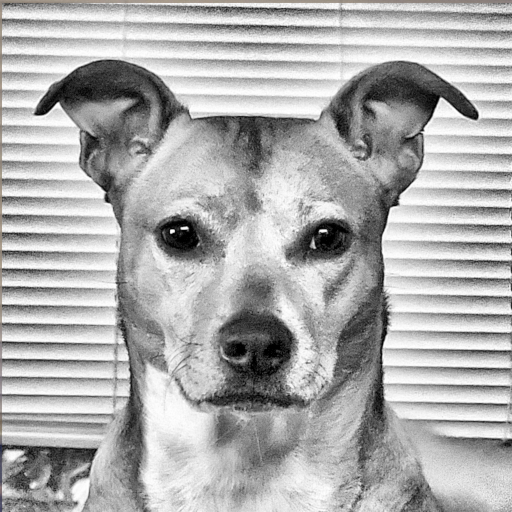}\hspace{0.1cm}
	\includegraphics[width=0.25\linewidth]{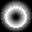} \\[1ex]
	\includegraphics[width=0.25\linewidth]{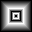}\hspace{0.1cm}
	\includegraphics[width=0.25\linewidth]{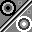}\hspace{0.1cm}
	\includegraphics[width=0.25\linewidth]{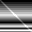}
	\caption{Six different images used in experiments for the image reconstruction problem \eqref{eq:image-reconstruction}}
	\label{fig:images}
\end{figure}

For the trust-region subproblem, the authors in \cite{manns2024discrete} use $\alpha_i = 4(\sqrt{5})^i\times 10^{-4}$ for $i=0,1,2,3,4$ for $N=32$ and $N=64$. For $N=96$, they take $\alpha_i$ for $i=1,2,3$. The choice of $\Delta$ varies as the trust-region algorithms proceed, expanding and shrinking, and it starts at $\frac{1}{16}N^2$. In practice, this varies from $\Delta = 1$ up to $\Delta = 576$ for the problems on the largest graphs. The total number of instances for each problem is listed in the final column in Table \ref{tab:trust-region-wins}. 

For each instance, we ran Algorithm \ref{alg:Randomized-Augmentation} with various initial points and two choices for $p$. We consider the three distinct choices for an initial point as described in Section \ref{sec:starting-point}, namely, the minimal budget point (denoted ``Zero'' as these are zero budget points), the minimizer of \eqref{eq:penalty-formulation} found by a bisection algorithm (denoted ``Penalized''), and the rounded solution to the continuous relaxation (denoted ``Rounded''). We compute the minimizer of the penalized form \eqref{eq:penalty-formulation} using Algorithm \ref{alg:graver-augmentation}, and we perform a bisection algorithm to find optimal $\mu$ within an accuracy of $10^{-4}$. To find the solution to the continuous relaxation, we use CPLEX for both problems. The two choices that we use for $p$  are $p=0$ and $p=1$. We note that taking $p=0$ corresponds to always choosing the entering variable with most negative coefficient, and so we denote this by ``smallest coefficient'';  $p=1$ is analogous to the steepest edge rule when taking $H(x) = x$, and so we denote this case ``steepest edge.''

For each instance and configuration, we ran 100 trials of Algorithm \ref{alg:Randomized-Augmentation}, stopping either once 100 trials were completed or a timeout of 600 seconds was reached. In addition, each individual trial was timed out after 60 seconds. The solution given after each trial was subsequently polished to ensure that the solution is 2-optimal. We then chose the solution with minimal objective from all trials that completed. For comparison, we also used the solver CPLEX with default settings on each instance. We timed out the solve after 600 seconds and report the last solution found.
All experiments were performed on an Intel Xeon CPU with 192 cores, 2.2 GHz, and 1.48 TB of memory. We take advantage of the multiple cores to run the experiments in parallel, but each individual instance is solved in serial on a single core (meaning all 100 trials of Algorithm \ref{alg:Randomized-Augmentation} are performed sequentially).

\subsection{Results}

We record the results of the experiments averaged over all instances for given $N$ and $\alpha$ in various tables in this section. We considered three primary metrics by which to compare the different algorithms:  the number of instances in which a given method finds the best solution out of all methods, the MIP gap of the best solution found averaged over all instances, and the average time to completion. For the image reconstruction problem \eqref{eq:image-reconstruction}, we record the table of average solution times in Table \ref{tab:image-times}, the table counting the best methods in Table \ref{tab:image-wins}, and the table of the average MIP gaps in Table \ref{tab:image-gaps}.
For the trust-region--constrained problem \eqref{eq:trust-region-subproblem}, we record the average solution times in Table \ref{tab:trust-region-times}, the count of best methods for each class of instances in Table \ref{tab:trust-region-wins}, and the average MIP gaps in Table \ref{tab:trust-region-gaps}.

\begin{table}[]
    \centering
\begin{tabular}{cr|r|rrr|rrr}
 &  &       & \multicolumn{6}{c}{Algorithm \ref{alg:Randomized-Augmentation}} \\
 &  &       & \multicolumn{3}{c}{Smallest Coefficient $(p=0)$} & \multicolumn{3}{|c}{Steepest Edge $(p=1)$} \\
 $N$ & $i$ & CPLEX & Zero & Penalized & Rounded & Zero & Penalized & Rounded \\
\hline
\multirow[t]{4}{*}{32} & 1 & 600.0 & 2.3 & 0.9 & 1.4 & 2.2 & 1.0 & 1.4 \\
 & 2 & 600.0 & 2.3 & 0.9 & 1.6 & 2.5 & 1.0 & 1.5 \\
 & 3 & 600.0 & 2.5 & 1.0 & 1.7 & 3.2 & 1.3 & 2.2 \\
 & 4 & 550.2 & 3.0 & 1.2 & 2.1 & 5.3 & 2.0 & 3.5 \\
\hline
\multirow[t]{4}{*}{64} & 1 & 600.0 & 13.1 & 3.2 & 7.2 & 9.3 & 3.3 & 5.4 \\
 & 2 & 600.0 & 13.5 & 3.3 & 7.4 & 10.8 & 3.5 & 6.3 \\
 & 3 & 600.0 & 14.3 & 3.5 & 7.8 & 14.7 & 4.1 & 8.7 \\
 & 4 & 600.0 & 14.4 & 4.1 & 8.8 & 34.6 & 14.9 & 23.4 \\
\hline
\multirow[t]{4}{*}{96} & 1 & 600.0 & 48.1 & 7.1 & 22.5 & 23.5 & 7.2 & 12.8 \\
 & 2 & 600.0 & 48.4 & 7.2 & 23.2 & 27.8 & 7.9 & 16.0 \\
 & 3 & 600.0 & 46.7 & 7.8 & 22.5 & 38.9 & 9.3 & 23.0 \\
 & 4 & 600.0 & 43.1 & 9.0 & 22.8 & 68.5 & 15.0 & 40.1 \\
\end{tabular}
    \caption{Time to completion in seconds averaged over all instances of the image reconstruction problem \eqref{eq:image-reconstruction} on an $N\times N$ grid graph with TV regularization parameter $\alpha = 0.25i$. All solves were timed out after 600 seconds.}
    \label{tab:image-times}
\end{table}

In all but two cases in the image reconstruction problem, CPLEX fails to produce a solution that it deems optimal in the allotted time of 600 seconds. The remaining methods all finish one to two orders of magnitude faster in the worst case, needing slightly over a minute on average to complete 100 trials, as seen in the steepest edge zero parameterization. We observe generally that the instances take more time to complete as $\alpha$ increases and that this situation is especially pronounced when using the steepest edge rule. We suspect that this is largely due to its being more difficult to find improving moves as $\alpha$ increases. The impact of choosing a different starting point is also pronounced because taking the rounded solution has the effect of nearly halving the solution time whereas the penalized initialization leads to a reduction by a factor of 3 to 4 in the average solve time. 

\begin{table}[]
    \centering
\begin{tabular}{cr|r|rrr|rrr|c}
 &  &       & \multicolumn{6}{c|}{Algorithm \ref{alg:Randomized-Augmentation}} & \\
 &  &       & \multicolumn{3}{c}{Smallest Coefficient $(p=0)$} & \multicolumn{3}{|c|}{Steepest Edge $(p=1)$} \\
 $N$ & $i$ & CPLEX & Zero & Penalized & Rounded & Zero & Penalized & Rounded & Total\\
\hline
\multirow[t]{4}{*}{32} & 1 & \textbf{22} & 12 & 20 & 12 & 13 & 21 & 12 & 24\\
 & 2 & 19 & 10 & 18 & 10 & 15 & \textbf{21} & 14 & 24\\
 & 3 & \textbf{21} & 7 & 15 & 7 & 20 & \textbf{21} & 20 & 24\\
 & 4 & 18 & 7 & 17 & 6 & \textbf{22} & 21 & \textbf{22} & 24\\
\hline
\multirow[t]{4}{*}{64} & 1 & 9 & 7 & \textbf{17} & 6 & 9 & 16 & 7 & 24\\
 & 2 & 12 & 6 & 13 & 6 & 8 & \textbf{21} & 9 & 24\\
 & 3 & 13 & 6 & 11 & 6 & 15 & \textbf{21} & 13 & 24\\
 & 4 & 12 & 7 & 10 & 6 & 18 & 18 & \textbf{19} & 24\\
\hline
\multirow[t]{4}{*}{96} & 1 & 10 & 7 & 15 & 7 & 8 & \textbf{16} & 8 & 24\\
 & 2 & 8 & 7 & 10 & 7 & 9 & \textbf{19} & 10 & 24\\
 & 3 & 7 & 6 & 12 & 6 & 13 & \textbf{22} & 12 & 24\\
 & 4 & 8 & 6 & 11 & 7 & 15 & \textbf{20} & 15 & 24\\

\end{tabular}
    \caption{Number of instances for which each method produced the best (not necessarily optimal) solution among all methods in the image reconstruction problem \eqref{eq:image-reconstruction} on an $N\times N$ grid graph with TV regularization parameter $\alpha=0.25i$. The total number of instances is listed in the last column.}
    \label{tab:image-wins}
\end{table}

\begin{table}[]
    \centering
\begin{tabular}{cr|r|rrr|rrr}
 &  &       & \multicolumn{6}{c}{Algorithm \ref{alg:Randomized-Augmentation}} \\
 &  &       & \multicolumn{3}{c}{Smallest Coefficient $(p=0)$} & \multicolumn{3}{|c}{Steepest Edge $(p=1)$} \\ $N$ & $i$ & CPLEX & Zero & Penalized & Rounded & Zero & Penalized & Rounded \\
\hline
\multirow[t]{4}{*}{32} & 1 & \textbf{17.1} & (33.4)17.2 & (17.3)17.1 & (26.8)17.2 & (17.9)17.1 & (17.1)17.1 & (17.7)17.1 \\
 & 2 & 12.8 & (27.1)13.4 & (13.0)12.8 & (21.8)13.2 & (12.9)12.8 & (12.8)\textbf{12.7} & (13.0)12.8 \\
 & 3 & 9.23 & (22.4)10.7 & (9.67)9.29 & (17.6)10.4 & (9.31)9.25 & (9.23)\textbf{9.22} & (9.32)9.24 \\
 & 4 & 6.81 & (17.4)9.48 & (6.98)6.84 & (12.4)8.62 & (6.79)\textbf{6.77} & (6.78)6.78 & (6.79)\textbf{6.77} \\
\hline
\multirow[t]{4}{*}{64} & 1 & 21.9 & (42.1)22.1 & (22.1)21.8 & (33.1)22.1 & (22.6)22.0 & (21.9)\textbf{21.8} & (22.4)21.9 \\
 & 2 & 17.4 & (37.1)18.6 & (17.7)17.5 & (28.4)18.3 & (17.8)17.6 & (17.4)\textbf{17.4} & (17.7)17.5 \\
 & 3 & 13.6 & (31.8)16.3 & (13.9)13.7 & (24.1)15.4 & (13.7) 13.6 & (13.6)\textbf{13.6} & (13.7)13.7 \\
 & 4 & 10.5 & (26.3)15.0 & (11.1)10.6 & (20.0)13.1 & (10.4)10.4 & (10.4)10.4 & (10.4)\textbf{10.4} \\
\hline
\multirow[t]{4}{*}{96} & 1 & 24.1 & (45.3)24.2 & (24.2)23.9 & (35.7)24.3 & (24.6)24.0 & (23.9)\textbf{23.9} & (24.4)24.0 \\
 & 2 & 20.1 & (41.3)21.1 & (20.3)20.0 & (31.5)20.8 & (20.3)20.1 & (19.9)\textbf{19.9} & (20.3)20.0 \\
 & 3 & 15.9 & (36.2)18.6 & (15.9)15.6 & (27.0)17.5 & (15.7)15.6 & (15.6)\textbf{15.6} & (15.7)15.6 \\
 & 4 & 11.7 & (29.3)16.2 & (12.4)11.7 & (22.9)14.2 & (11.6)11.6 & (11.6)\textbf{11.6} & (11.6)11.6 \\
\end{tabular}
    \caption{Relative percentage MIP gaps for the image reconstruction problem  \eqref{eq:image-reconstruction} averaged over all instances on an $N\times N$ grid graph with TV regularization parameter $\alpha = 0.25i$. We compute this as $(obj - lb)/|obj| \times 100$, where $obj$ is the computed objective and $lb$ is the final lower bound found by CPLEX. Each entry for Algorithm 4 presents the gap before polishing in parentheses and the gap after polishing afterwards.}
    \label{tab:image-gaps}
\end{table}

In terms of solution quality, we observe that, generally, methods that make use of steepest edge pivoting produce the solutions with lowest objective; of those, when we initialize the solution from the penalized formulations, we obtain the best solutions. We also observe that polishing the solutions to be 2-optimal improves the solutions in almost all cases, but the improvement is significant when using smallest coefficient pivoting (often reducing the MIP gap by a factor of 2). For small $N$, we observe that CPLEX performs comparably with the best methods, but the advantage in using randomized approaches becomes clear as $N$ increases. There is also a trend where the randomized approaches appear to perform better for higher $\alpha$. This may be explained by the large number of available improving Graver moves for low $\alpha$ that decreases the probability of finding a correct sequence of moves to the optimum. For larger $\alpha$, solutions  generally comprise large connected regions of constant value, meaning fewer moves are needed, signifying a higher probability of success.
A sample depiction of the solutions produced by each of these methods is given in the appendix in Figure \ref{fig:image-panel}.

\begin{table}[]
    \centering
\begin{tabular}{cr|r|rrr|rrr}
 &  &       & \multicolumn{6}{c}{Algorithm \ref{alg:Randomized-Augmentation}} \\
 &  &       & \multicolumn{3}{c}{Smallest Coefficient $(p=0)$} & \multicolumn{3}{|c}{Steepest Edge $(p=1)$} \\ $N$ & $i$ & CPLEX & Zero & Penalized & Rounded & Zero & Penalized & Rounded \\
\hline
\multirow[t]{5}{*}{32} & 0 & 0.7 & 0.7 & 0.7 & 0.7 & 1.3 & 1.2 & 1.2 \\
 & 1 & 1.6 & 0.9 & 0.9 & 0.9 & 1.7 & 1.6 & 1.6 \\
 & 2 & 2.2 & 1.2 & 1.2 & 1.2 & 3.1 & 3.0 & 3.1 \\
 & 3  & 4.2 & 3.9 & 3.9 & 3.8 & 9.0 & 9.0 & 9.0 \\
 & 4 & 74.8 & 46.0 & 45.5 & 46.2 & 52.6 & 52.7 & 52.8 \\
\hline
\multirow[t]{5}{*}{64} & 0 & 66.2 & 9.2 & 9.2 & 9.2 & 40.3 & 40.0 & 40.1 \\
 & 1 & 46.8 & 15.2 & 15.5 & 15.3 & 49.8 & 49.1 & 49.2 \\
 & 2 & 208.5 & 23.4 & 23.1 & 23.3 & 76.1 & 75.9 & 75.8 \\
 & 3 & 279.0 & 170.4 & 166.8 & 167.3 & 278.9 & 277.0 & 277.9 \\
 & 4 & 600.0 & 600.0 & 600.0 & 600.0 & 600.0 & 600.0 & 600.0 \\
\hline
\multirow[t]{3}{*}{96} & 1 & 295.3 & 117.2 & 117.6 & 117.8 & 281.1 & 280.4 & 280.8 \\
 & 2 & 507.4 & 241.0 & 241.3 & 240.8 & 415.4 & 415.1 & 415.1 \\
 & 3 & 473.4 & 365.2 & 363.7 & 364.0 & 431.6 & 430.8 & 431.7 \\
\end{tabular}
    \caption{Time to completion in seconds averaged over all instances of the trust-region--constrained problem \eqref{eq:trust-region-subproblem} from \cite{manns2024discrete} on an $N\times N$ grid graph with TV regularization parameter $\alpha = 4(\sqrt{5})^i \times 10^{-4}$. All solves were timed out after 600 seconds.}
    \label{tab:trust-region-times}
\end{table}

The trust-region subproblem instances proved to be harder to solve for the randomized algorithms. In particular,  in almost all cases, the different choices of initialization provided no benefit in that both the rounded and penalized initializations were almost always the same as the minimal budget point. In addition, we observe that on average these algorithms failed to terminate in the allotted time for several of the largest problems. For the steepest edge parameterizations, 20 of the instances for $N=96$ failed to find any solution better than the initial point before timing out. On average we observed that the largest coefficient configurations tended to terminate between one to two times as fast as CPLEX; the steepest edge parameterizations were slightly faster but comparable in total runtime to CPLEX.

\begin{table}[]
    \centering
\begin{tabular}{cr|r|rrr|rrr|c}
 &  &       & \multicolumn{6}{c|}{Algorithm \ref{alg:Randomized-Augmentation}} \\
 &  &       & \multicolumn{3}{c}{Smallest Coefficient $(p=0)$} & \multicolumn{3}{|c|}{Steepest Edge $(p=1)$} \\ $N$ & $i$ & CPLEX & Zero & Penalized & Rounded & Zero & Penalized & Rounded & Total\\
\hline
\multirow[t]{5}{*}{32} & 0 & \textbf{40} & 33 & 34 & 34 & 37 & 35 & 35 & 40\\
 & 1 & \textbf{34} & 26 & 27 & 27 & 31 & 29 & 29 & 34 \\
 & 2 & \textbf{24} & 19 & 19 & 19 & 23 & 23 & 23 & 24\\
 & 3 & \textbf{15} & 14 & 14 & 14 & 13 & 13 & 13 & 15 \\
 & 4 & \textbf{1} & \textbf{1} & \textbf{1} & \textbf{1} & \textbf{1} & \textbf{1} & \textbf{1} & 1\\
\hline
\multirow[t]{5}{*}{64} & 0 & \textbf{68} & 36 & 39 & 39 & 60 & 56 & 56 & 70\\
 & 1 & \textbf{38} & 23 & 24 & 24 & 33 & 32 & 32 & 38\\
 & 2 & 44 & 34 & 35 & 35 & \textbf{51} & \textbf{51} & \textbf{51} & 53\\
 & 3 & 15 & 14 & 14 & 14 & 15 & \textbf{16} & \textbf{16} & 19 \\
 & 4 & \textbf{1} & \textbf{1} & \textbf{1} & \textbf{1} & \textbf{1} & \textbf{1} & \textbf{1} & 1\\
\hline
\multirow[t]{3}{*}{96} & 1 & \textbf{58} & 48 & 51 & 51 & 49 & 49 & 49 & 74\\
 & 2 & 52 & \textbf{138} & \textbf{138} & \textbf{138} & 109 & 109 & 109 & 173\\
 & 3 & 13 & \textbf{30} & 29 & 29 & 21 & 21 & 21 & 35\\

\end{tabular}
    \caption{Number of instances from \cite{manns2024discrete} for which each method produced the best (not necessarily optimal) solution on an $N\times N$ grid graph with TV regularization parameter $\alpha = 4(\sqrt{5})^i \times 10^{-4}$ in the trust-region--constrained problem \eqref{eq:trust-region-subproblem}. The final column contains the total number of instances.}
    \label{tab:trust-region-wins}
\end{table}

\begin{table}[]
    \centering
\begin{tabular}{cr|r|rrr|rrr}
 &  &       & \multicolumn{6}{c}{Algorithm \ref{alg:Randomized-Augmentation}} \\
 &  &       & \multicolumn{3}{c}{Smallest Coefficient $(p=0)$} & \multicolumn{3}{|c}{Steepest Edge $(p=1)$} \\ $N$ & $i$ & CPLEX & Zero & Penalized & Rounded & Zero & Penalized & Rounded \\
\hline
\multirow[t]{5}{*}{32} & 0 & \textbf{0.00} & (1.13)0.11 & (0.47)0.02 & (0.47)0.02 & (0.04)0.01 & (0.03)0.02 & (0.03)0.02 \\
 & 1 & \textbf{0.00} & (3.66)2.65 & (1.23)0.44 & (1.23)0.44 & (0.08)0.03 & (0.14)0.09 & (0.14)0.09 \\
 & 2 & \textbf{0.00} & (1.34)0.33 & (1.34)0.33 & (1.34)0.33 & (0.12)0.03 & (0.12)0.03 & (0.12)0.03 \\
 & 3 & \textbf{0.00} & (39.9)0.05 & (39.9)0.05 & (39.9)0.05 & (15.1)11.0 & (15.1)11.0 & (15.1)11.0 \\
 & 4 & \textbf{0.00} & (0.00)0.00 & (0.00)0.00 & (0.00)0.00 & (0.00)0.00 & (0.00)0.00 & (0.00)0.00 \\
\hline
\multirow[t]{5}{*}{64} & 0 & \textbf{0.21} & (1.42)0.62 & (1.08)0.61 & (1.08)0.61 & (0.25)0.24 & (0.25)0.24 & (0.25)0.24 \\
 & 1 & \textbf{0.00} & (3.96)1.67 & (2.78)0.99 & (2.78)0.99 & (0.04)0.01 & (0.06)0.02 & (0.06)0.02 \\
 & 2 & 0.78 & (1.85)1.25 & (1.48)1.03 & (1.48)1.03 & (0.67)\textbf{0.64} & (0.67)\textbf{0.64} & (0.67)\textbf{0.64} \\
 & 3 & 4.99 & (6.24)4.94 & (5.88)4.93 & (5.88)4.93 & (4.93)\textbf{4.78} & (4.92)\textbf{4.78} & (4.92)\textbf{4.78} \\
 & 4 & --- & (---) --- & (---) --- & (---) --- & (---) --- & (---) --- & (---) --- \\
\hline
\multirow[t]{3}{*}{96} & 1 & \textbf{2.55} & (5.24)3.86 & (4.86)3.50 & (4.86)3.50 & (2.80)2.65 & (2.78)2.65 & (2.78)2.65 \\
 & 2 & 6.90 & (6.04)\textbf{5.73} & (6.04)\textbf{5.73} & (6.04)\textbf{5.73} & (7.80)7.76 & (7.80)7.76 & (7.80)7.76 \\
 & 3 & 16.5 & (15.0)\textbf{14.7} & (15.2)14.7 & (15.2)14.7 & (16.0)15.7 & (15.8)15.5 & (15.7)15.4 \\

\end{tabular}
    \caption{Relative percentage MIP gaps for the trust-region--constrained problem \eqref{eq:trust-region-subproblem} averaged over all instances from \cite{manns2024discrete} on an $N\times N$ grid with TV regularization parameter $\alpha = 4(\sqrt{5})^i \times 10^{-4}$. We compute this as $(obj - lb)/|obj| \times 100$, where $obj$ is the computed objective and $lb$ is the final lower bound found by CPLEX. For $N=64$ and $i = 4$, the best solutions found took value 0, and so a relative gap could not be computed. Each entry for Algorithm 4 presents the gap before polishing in parentheses and the gap after polishing afterwards.}
    \label{tab:trust-region-gaps}
\end{table}

In terms of quality of solutions, for the smaller instances with $N=32$, CPLEX almost always produces the best solution, which is apparent from the average MIP gap produced for each instance. The randomized methods, save for a few instances, are generally within one percent of the optimal solution for these cases. For the larger instances and for higher $\alpha$, the randomized instances outperformed CPLEX. The steepest edge configurations perform marginally better in terms of MIP gap for higher $\alpha$ when $N=64$; and for $N=96$ and high $\alpha$, the smallest coefficient randomized methods each perform one to two percentage points better on average. The steepest edge parameterizations also perform better than CPLEX, but, as mentioned before, in several instances  this method timed out before finding any good solutions.

As a whole, these randomized approaches can produce high-quality solutions that are near optimal in many cases. For smaller instances where CPLEX can solve to optimality, the solutions these methods produce are often within a one percent gap; and for the larger instances, these methods often outperform CPLEX significantly.
Empirically, the computational time for the randomized heuristics also appears to scale linearly, making the approaches more suitable for larger problems.

\subsection{Relation between $p$, Objective, and Iteration Count}
\label{sec:p-relation}

\begin{figure}
    \centering
    \includegraphics[width=\linewidth]{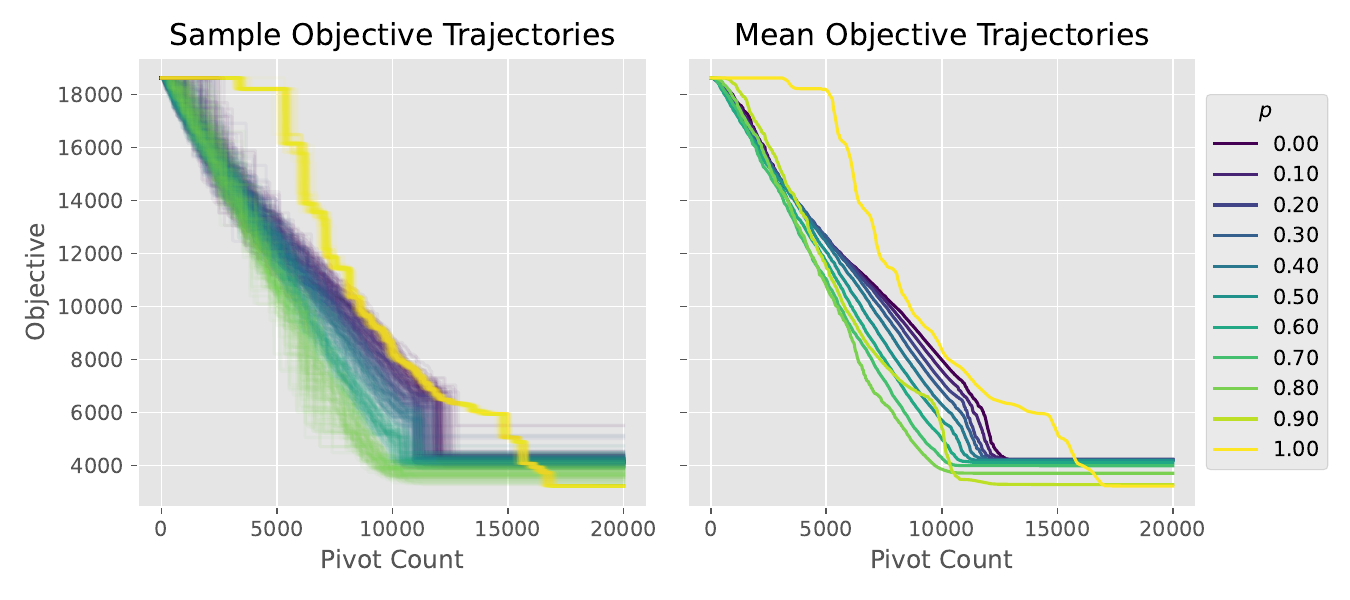}
    \caption{Objective trajectories with varying $p$ for image reconstruction problem \eqref{eq:image-reconstruction} on Image 4 with $N=64, \alpha=2, \delta=0.75$. The left plot contains 100 sample trajectories for $p \in \{0,0.2,\ldots,1\}$ and the right plots the average objective after a given number of pivots for $p \in \{0,0.1,\ldots,1\} $. The large vertical drop in objectives for most trajectories near pivot 11000 corresponds to the change following solution polishing.}
    \label{fig:objective-trajectories}
\end{figure}

In this subsection we briefly remark on some interesting qualitative results we observed regarding the procedure of the algorithm as we varied the parameter $p$ introduced in \eqref{eq:weighting-scheme} in Algorithm \ref{alg:Randomized-Augmentation}, which tunes how much pivoting is influenced by change in $H$. In the main group of experiments, we considered solely $p=0$ and $p=1$ and largely observed that setting $p=1$ produces solutions of a lower objective, but with significantly longer runtimes. As we adjust $p$ between 0 and 1 and beyond, we observe an interesting nonlinear relationship between $p$ and both the objective and number of pivots that merits further study.

In Figure \ref{fig:objective-trajectories} we depict sample objective trajectories for a specific instance of the image reconstruction problems \eqref{eq:image-reconstruction}. We note that there appears to be a qualitative shift in algorithm behavior as $p$ approaches 1, where for small $p$ the trajectories largely bear the same shape and exhibit a larger variance in objective, suggesting that a greater diversity of Graver augmentations is selected. For $p$ close to 1, there is lower variance in objective, and there are larger discrete jumps in objective, suggesting there are similar Graver moves being taken among the sample at around the same time. These moves also on average take more time to find. 

\begin{figure}
    \centering
    \includegraphics[width=\linewidth]{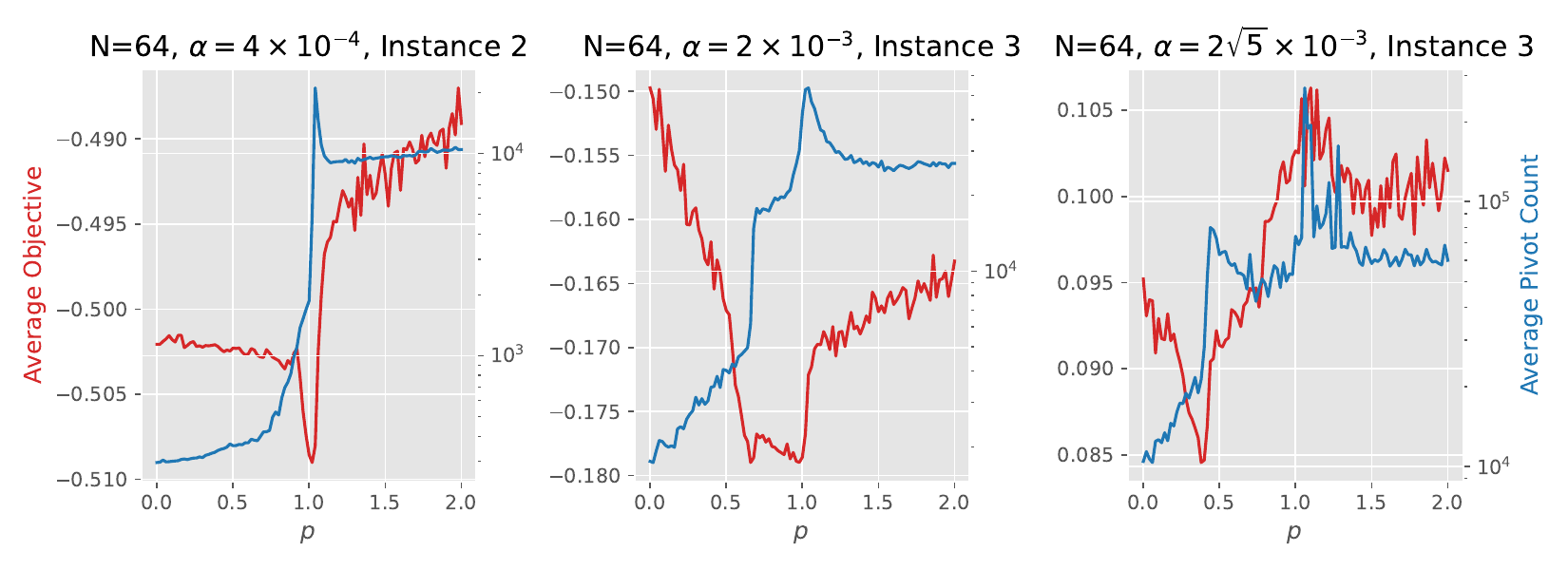}
    \caption{Relationship between average objective (in red), average number of pivots (in blue), and $p$ in Algorithm \ref{alg:Randomized-Augmentation} for three select instances of the trust-region--constrained problem \eqref{eq:trust-region-subproblem}. Quantities are averaged over 100 trials of the algorithm.}
    \label{fig:efficiency-transition}
\end{figure}

This kind of qualitative shift in behavior is also apparent in the trust-region--constrained problems. In Figure \eqref{fig:efficiency-transition} we plot the relationship between $p$ and the average objective and the average number of pivots made. As can be seen in each of the plots, there is a clear transition where the number of pivots significantly increases, in these cases by one to two orders of magnitude, and the quality of the solutions significantly improves as well. Interestingly, while this phase transition often occurs around $p=1$, there are specific instances where it occurs earlier and where choosing $p=1$ gives the worst solutions on average among all choices for $p$. There also appears to be a spike in the number of pivots for $p$ slightly larger than $p=1$. Given the clear improvement in objective occurring around this phase transition, it stands to reason that the performance of these randomized algorithms can be significantly improved by properly choosing $p$. We leave a more in-depth study on this choice to future work.

\section{Conclusion}
\label{sec:conclusion}
In this paper we establish the structure of the Graver basis for a class of problems derived from total variation-regularized optimization problems, and we demonstrate how to exploit this structure to devise efficient globally optimal augmentation algorithms.
In spite of the typically exponential size of the basis, we show how to efficiently search this set for improving moves. 
We then demonstrate how to implement a Graver augmentation algorithm for problems with additional constraints via a randomized variant of the simplex method.
We demonstrate empirically that for large problems this method can outperform state-of-the-art integer program solvers.

Our discussion of the randomized augmentation algorithm has left several aspects  untouched.
In particular, it would be of interest to assess whether the number of pivots in the randomized simplex method may be bounded and how close to optimal the solutions produced are.
We would also like to study further the choice of $p$ in Algorithm \ref{alg:Randomized-Augmentation} since in Section \ref{sec:p-relation} we observed clear phase transitions in the performance of the algorithms at specific values of $p$.
The success in using randomized Graver augmentation for this class of problems might also translate to other combinatorial optimizations.
In particular, the Graver basis for matching problems is well known to correspond to specific alternating walks \cite{reyes2012minimal}. 
We may be able to address various constrained versions of these classical problems via randomized construction of Graver elements.
The methods presented in this paper also naturally extend to problems with multiple separable constraints; and as in \cite{manns2024discrete}, we may situate these randomized subproblem solves in a trust-region algorithm for problems with general nonlinear $F$.
    % \subsection{Hypergraphs: Graver bases and Sampling}
    % \subsection{Generalization to Partially Separable Problems}
    % \Miles{Matching problems can be mentioned here.}
\section*{Data Availability}
The software developed for performing the experiments discussed in this paper and the instances of the image reconstruction problem are available from the corresponding author upon request. The trust-region--constrained subproblem instances were provided by the authors of \cite{manns2024discrete} and can be found at the following repository: \url{https://github.com/INFORMSJoC/2024.0680}.

\section*{Acknowledgments}
This work was supported by the U.S. Department of Energy, Office of Science, Advanced Scientific Computing Research, under contract number DE-AC02-06CH11357.

\section*{Conflicts of Interest}
The authors have no competing interests to declare that are relevant to the content of this article.

\bibliography{refs}
%\printbibliography
\vfill
\begin{flushright}
\scriptsize
\framebox{\parbox{\textwidth}{
The submitted manuscript has been created by UChicago Argonne, LLC, Operator of Argonne National Laboratory (“Argonne”). 
Argonne, a U.S. Department of Energy Office of Science laboratory, is operated under Contract No. DE-AC02-06CH11357. 
The U.S. Government retains for itself, and others acting on its behalf, a paid-up nonexclusive, irrevocable worldwide 
license in said article to reproduce, prepare derivative works, distribute copies to the public, and perform publicly 
and display publicly, by or on behalf of the Government.  The Department of Energy will provide public access to these 
results of federally sponsored research in accordance with the DOE Public Access Plan. 
\url{http://energy.gov/downloads/doe-public-access-plan}.
}}
\normalsize
\end{flushright}

\pagebreak

\appendix

\section{Additional Experimental Results}

We also present the best solutions provided by each approach on each sample image for the $96\times 96$ image for $\alpha = 1$ and $\delta = 0.75$ in Figure \ref{fig:image-panel}. There are clear qualitative differences in the solutions given by choosing different pivot rules for the randomized algorithms. In particular, we observe that for the approaches using smallest coefficient, the solutions clearly exhibit higher total variation in that there are more speckle patterns, particularly in the solutions initialized with the zero-budget solution and the rounded solution. Algorithms that use the steepest edge approach genuinely appear to make more efficient usage of the budget. This is especially apparent for Image 2 where the solutions are more square whereas for the largest coefficient runs, random protrusions about the square appear. In terms of objective, as seen in Table \ref{tab:image-wins}, initialization using the penalized problem and then applying the steepest edge pivot rule produces the best solution in terms of objective, and this can be seen qualitatively in these figures.

\begin{figure}
    \centering
    \includegraphics[width=\linewidth]{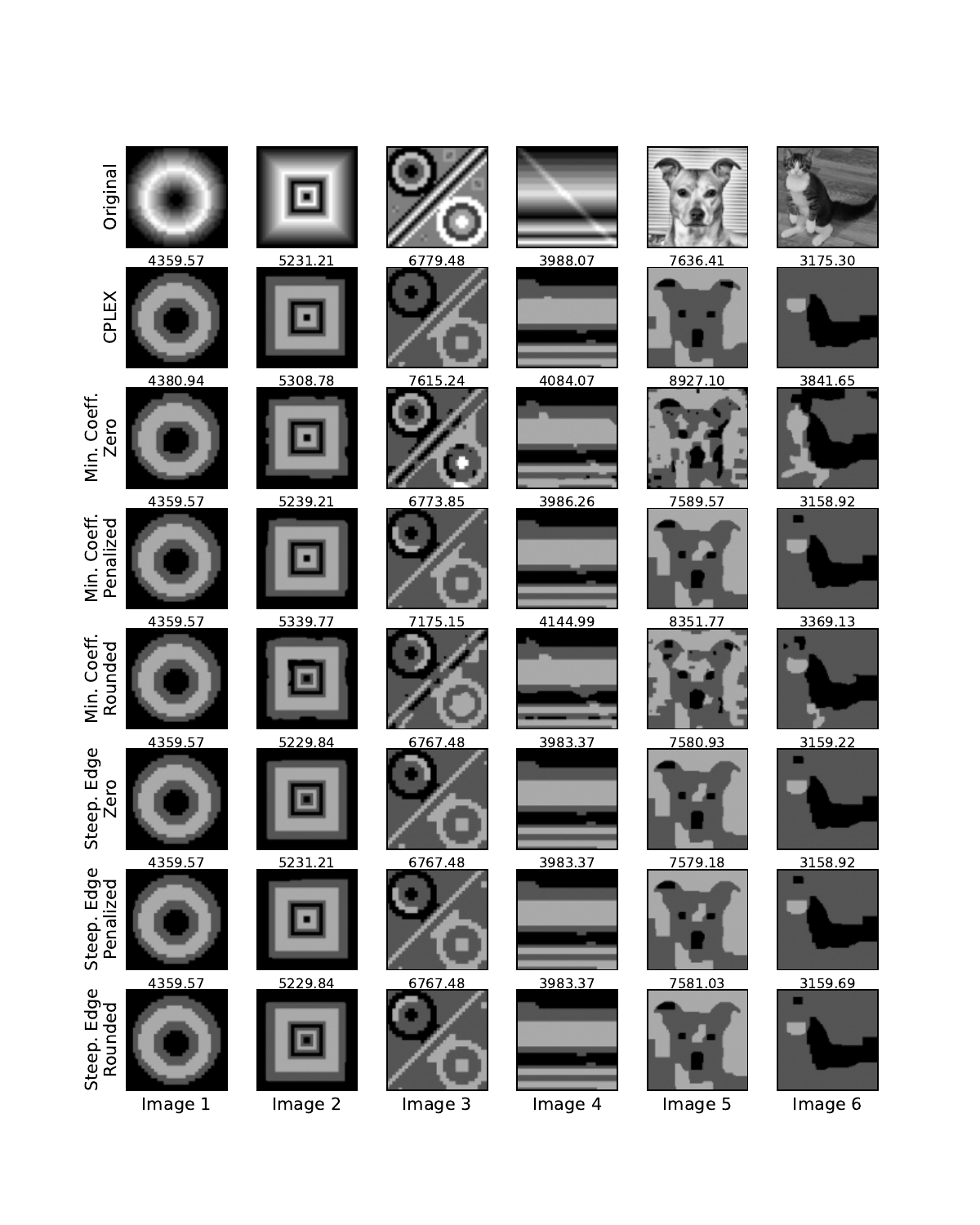}
    \caption{Selection of test images in row 1 followed by the best solutions produced by seven approaches to finding the solution to the image reconstruction problem \eqref{eq:image-reconstruction} on an $96\times 96$ grid with $Q=3$, $\alpha = 2$, and $\delta = 0.75$. The number above an image denotes the objective value of the given solution.}
    \label{fig:image-panel}
\end{figure}

\end{document}